\documentclass[10pt,leqno]{article}
\usepackage{latexsym,amsfonts,amsmath,amssymb,amsthm,url}
\usepackage[english]{babel}
\usepackage[latin1]{inputenc}
\usepackage{amsbsy}
\usepackage{amscd}
\usepackage{graphicx,psfrag}
\usepackage{color}
\usepackage{enumerate}
\usepackage{dsfont}
\newtheorem{thm}{Theorem}
\newtheorem{prop}[thm]{Proposition}
\newtheorem{lem}[thm]{Lemma}
\newtheorem{cor}[thm]{Corollary}
\newtheorem{rem}[thm]{Remark}

\newcommand{\beqn}{\begin{equation}}
\newcommand{\eeqn}{\end{equation}}
\newcommand{\bear}{\begin{eqnarray}}
\newcommand{\eear}{\end{eqnarray}}
\newcommand{\bean}{\begin{eqnarray*}}
\newcommand{\eean}{\end{eqnarray*}}
\renewcommand\qed{\hfill $\square$}
\newcommand{\NN}{\mathbb N}

\renewcommand{\P}{\mathbb{P}}
\renewcommand{\L}{\mathcal{L}}

\providecommand{\E}{\mathbb{E}}
\providecommand{\F}{\mathcal{F}}

\providecommand{\N}{\mathbb{N}}

\newcommand{\e}{\varepsilon}
\newcommand{\ee}{\mathbf{e}}

\def\Xun{\mathcal{X}_{1,1}}
\def\Xuni{\mathcal{X}_{1,i}}
\begin{document}

\title{\textbf{A stochastic min-driven coalescence process and its hydrodynamical limit}}

\author{
\normalsize{
Anne-Laure \textsc{Basdevant}\footnotemark[1], Philippe \textsc{Lauren\c cot}\footnotemark[1], James R. \textsc{Norris}\footnotemark[2]  \;and Cl\'ement \textsc{Rau}\footnotemark[1]}
}
\footnotetext[1]{Institut de Math\'ematiques de Toulouse, CNRS UMR~5219, Universit\'e de Toulouse, F--31062 Toulouse cedex 9, France. E-mail: {\tt abasdeva@math.ups-tlse.fr, laurenco@mip.ups-tlse.fr, rau@math.ups-tlse.fr}}
\footnotetext[2]{Statistical Laboratory, Centre for Mathematical Sciences, Wilberforce Road, Cambridge CB3 0WB, United Kingdom. E-mail: {\tt j.r.norris@statslab.cam.ac.uk}}
\date{\today}
\maketitle \vspace{-1cm}

\vspace*{0.2cm}

\begin{abstract} A stochastic system of particles is considered in which the sizes of the particles increase by successive binary mergers with the constraint that each coagulation event involves a particle with minimal size. Convergence of a suitably renormalised version of this process to a deterministic hydrodynamical limit is shown and the time evolution of the minimal size is studied for both deterministic and stochastic models.
\end{abstract}

\bigskip
{\small{

\bigskip
\noindent{\bf A.M.S. Classification. } 82C22, 60K35, 60H10, 34A34, 34C11.

\bigskip

\section{Introduction}
\setcounter{thm}{0}
\setcounter{equation}{0}
Coagulation models describe the evolution of a population of particles increasing their sizes by successive binary mergers, the state of each particle being fully determined by its size. Well-known examples of such models are the Smoluchowski coagulation equation \cite{Sm16,Sm17} and its stochastic counterpart, the Marcus-Lushnikov process \cite{Lu78,Ma68}, and both have been extensively studied in recent years (see \cite{Al99,Be06,LM04,Lv03,No99,Wa06} and the references therein). Another class of coagulation models has also received some interest, the main feature of these models being that the particles with the smallest size play a more important role than the others. A first example are the Becker-D\"oring equations: in that case, the (normalized) sizes of the particles range in the set of positive integers and a particle can only modify its size by gaining or shedding a particle with unit size \cite{BCP86}. Another example are the min-driven coagulation equations: given
 a positive integer $k$, at each step of the process, a particle with the smallest size $\ell$ is chosen and broken into $k$  daughter particles with size $\ell/k$, which are then pasted to other particles chosen at random in the population with equal probability \cite{CP00,DGY91,GM03,MNPxx}.

In this paper, we focus on the min-driven coagulation equation with $k=1$ (that is, there is no break-up of the particle of minimal size) but relax the assumption of deposition with equal probability. More specifically, the coalescence mechanism we are interested in is the following: consider an initial configuration $X=(X_i)_{i\ge 1}$ of particles, $X_i$ denoting the number of particles of size $i\ge 1$, and define the minimal size $\ell_X$ of $X$ as the smallest integer $i\ge 1$ for which $X_i>0$ (that is, $X_{\ell_X}>0$ and $X_i=0$ for $i\in\{1,\ldots,\ell_X-1\}$ if $\ell_X>1$). We pick a particle of size $\ell_X$, choose at random another particle of size $j\ge\ell_X$ according to a certain law, and merge the two particles to form a particle of size $\ell_X+j$. The system of particles thus jumps from the state $X$ to the state $Y=(Y_i)_{i\ge 1}$ given by $Y_k=X_k$ if $k\not\in\{\ell_X,j,\ell_X+j\}$ and
$$
\begin{array}{lllcl}
Y_{\ell_X} = X_{\ell_X}-1\,, & Y_j = X_j-1\,, & Y_{\ell_X+j}=X_{\ell_X+j}+1 & \mbox{ if } & j>\ell_X\,, \\
Y_{\ell_X} = X_{\ell_X}-2\,, & & Y_{2\ell_X}=X_{2\ell_X}+1 & \mbox{ if } & j=\ell_X\,,
\end{array}
$$
Observe that no matter is lost during this event. It remains to specify the probability of this jump to take place: instead of assuming it to be uniform and independent of the sizes of the particles involved in the coalescence event as in \cite{DGY91}, we consider the more general case where the jump from the state $X$ to the state $Y$ occurs at a rate $K(\ell_X,j)$, the coagulation kernel $K$ being a positive and symmetric function defined in $(\N\setminus\{0\})^2$.

A more precise description of the stochastic process is to be found in the next section, where a renormalized version of this process is also introduced. We will show that, as the total mass diverges to infinity, the renormalized process converges towards a deterministic limit which solves a countably infinite system of ordinary differential equations (Theorem~\ref{thc3}). The convergence holds true provided the coagulation kernel $K(i,j)$ does not increase too fast as $i,j\to\infty$, a typical example being
\begin{equation}
\label{obelix}
K(i,j) = \phi(i)\wedge\phi(j)\,, \quad i,j\ge 1\,, \;\;\;\mbox{ for some positive and non-decreasing function }\; \phi\,.
\end{equation}
Well-posedness of the system solved by the deterministic limit is also investigated (Theorem~\ref{thc1}) and reveals an interesting phenomenon, namely the possibility that the minimal size becomes infinite in finite time according to the growth of $K$ (Theorem~\ref{thc2}). Such a property also shows up for the stochastic min-driven coagulation process in a suitable sense (Theorem~\ref{thc4}). It is worth pointing out that coagulation kernels $K$ of the form \eqref{obelix} play a special role here.
%
\subsection{The stochastic min-driven coagulation process}\label{secdescrip}

We now describe more precisely the stochastic min-driven coagulation process to be studied in this paper. It is somehow reminiscent of the Marcus-Lushnikov process \cite{Lu78,Ma68} (which is related to the Smoluchowski coagulation equation). As in this process, two particles are chosen at random according to a certain law and merged but there is here an additional constraint; namely, one of the particles involved in the coalescence event has to be of minimal size among all particles in the system. To be more precise, we fix some positive integer $N$ and an initial condition $X_0^N=(X_{i,0}^N)_{i\ge 1}\in \ell_\NN^1$ such that
\begin{equation}\label{asterix}
\sum_{i=1}^\infty i\ X_{i,0}^N = N\,,
\end{equation}
where $X_{i,0}^N$ is the number of particles of size $i\ge 1$ and $\ell_\NN^1$ denotes the space of summable nonnegative and integer-valued sequences
\begin{equation}
\label{idefix}
\ell_\NN^1 := \left\{ X_0=(X_{i,0})_{i\ge 1} \in\ell^1(\NN\setminus\{0\}) \ : \ X_{i,0}\in\NN \;\;\;\mbox{ for all }\;\;\; i\ge 1\right\}\,.
\end{equation}
We next consider a time-dependent random variable $X^N(t)=(X_i^N(t))_{i\ge 1}$ which encodes the state of the process at time $t$ starting from the configuration $X_0^N$, its $i^{\hbox{\tiny{th}}}$-component $X_i^N(t)$ standing for the number of particles of size $i\ge 1$ at time $t\ge 0$. We assume that $X^N(0)=X_0^N$, so that $N$ is equal to the total mass initially present in the system. The process $(X^N(t))_{t\ge 0}$ evolves then as a Markov process with the following transition rules: if, at a time $t$, the process is in the state $X^N(t)=X=(X_i)_{i\ge 1}$ with minimal size $\ell_X\ge 1$ (that is, $X_{\ell_X}>0$ and $X_i=0$ for $1\le i \le \ell_X-1$ if $\ell_X>1$), only a given particle among the $X_{\ell_X}$ particles of minimal size $\ell_X$ can coalesce with another particle and this coagulation event occurs at the rate $K(\ell_X,j)$, where $j\ge\ell_X$ is the size of the second particle involved in the coagulation. Mathematically, this means that the process jumps f
 rom the state $X^N(t)=X$ to a state of the form
$$
Y=(0,\ldots,0,X_{\ell_X}-1,X_{\ell_X+1},\ldots,X_j-1,\ldots,X_{\ell_X+j}+1,\ldots) \;\;\;\mbox{ with rate }\;\;\; K(\ell_X,j) X_j
$$
for some $j>\ell_X$ or to the state
$$
Z=(0,\ldots,0,X_{\ell_X}-2,X_{\ell_X+1},\ldots,X_{2\ell_X}+1,\ldots) \;\;\;\mbox{ with rate }\;\;\; K(\ell_X,\ell_X) (X_{\ell_X}-1)\,.
$$
Equivalently, this means that the process waits an exponential time of parameter
$$
\lambda_X:=\left(\sum_{j=\ell_X}^\infty K(\ell_X,j) X_j\right) - K(\ell_X,\ell_X)
$$
and then jumps to the state $Y$ with probability $K(\ell_X,j) X_j/\lambda_X$ for $j>\ell_X$ and to the state $Z$ with probability $K(\ell_X,\ell_X) (X_{\ell_X}-1)/\lambda_X$. Observe that, as $X_{\ell_X}$ could be equal to $1$ or $2$, there might be no particle of size $\ell_X$ after this jump and the minimal size thus increases. In addition, we obviously have
$$
\sum_{i=1}^\infty i\ Y_i = \sum_{i=1}^\infty i\ Z_i = \sum_{i=1}^\infty i\ X_i\,,
$$
so that the total mass contained in the system of particles does not change during the jumps. Consequently,
\begin{equation}
\label{ielosubmarine}
\sum_{i=1}^\infty i\ X_i^N(t) = \sum_{i=1}^\infty i\ X_{i,0}^N = N \;\;\; \mbox{ for all }\;\;\; t\ge 0\,.
\end{equation}

\medskip

As already mentioned, one aim of this paper is to prove that, under some assumptions on the coagulation kernel $K$ and the initial data $(X_0^N)_{N\ge 1}$, a suitably renormalised version of the stochastic process converges to a deterministic limit as $N$ tends to infinity. More precisely, we introduce $\tilde{X}^N:=X^N/N$ and, for further use, list some properties of this process. Owing to the above construction, the generator $\L^N=(\L_k^N)_{k\ge 1}$ of this renormalised process  reads
\begin{eqnarray}
(\L^N_k f)(\xi) & = & N\ \left(\sum_{j=\ell_\xi}^\infty K(\ell_\xi,j)\ \xi_j\ \left[ f_k\left( \xi+\frac{\mathbf{e}_{\ell_\xi+j}}{N} - \frac{\mathbf{e}_{\ell_\xi}}{N} - \frac{\mathbf{e}_{j}}{N} \right) - f_k(\xi) \right] \right) \label{gen} \\
& & -\ K(\ell_\xi,\ell_\xi)\ \left[ f_k\left( \xi + \frac{\mathbf{e}_{2\ell_\xi}}{N} - 2\ \frac{\mathbf{e}_{\ell_\xi}}{N} \right)-f_k(\xi) \right]\,,\nonumber
\end{eqnarray}where $f=(f_k)_{k\ge 1}:\ell^1(\NN\setminus\{0\})\to\ell^1(\NN\setminus\{0\})$ and $(\mathbf{e}_i)_{i\ge 1}$ denotes the canonical basis of $\ell^1(\NN\setminus\{0\})$.
Moreover, the quadratic variation $\mathcal{Q}^N=(\mathcal{Q}_k^N)_{k\ge 1}$ of the martingale
$$
f(\tilde{X}^N(t))-\int_0^t(\L^N f)(\tilde{X}^N(s))ds
$$
is
\begin{eqnarray}
(\mathcal{Q}_k^N f)(\xi) & = & N\ \left(\sum_{j=\ell_\xi}^\infty K(\ell_\xi,j)\ \xi_j\ \left[ f_k\left( \xi + \frac{\mathbf{e}_{\ell_\xi+j}}{N} - \frac{\mathbf{e}_{\ell_\xi}}{N} - \frac{\mathbf{e}_{j}}{N} \right) -f_k\left( \xi \right) \right]^2\right) \label{QN}\\
& & -\ K(\ell_\xi,\ell_\xi)\ \left[ f_k\left( \xi+\frac{\mathbf{e}_{2\ell_\xi}}{N} - \frac{2\mathbf{e}_{\ell_\xi}}{N} \right)-f_k\left(\xi\right) \right]^2\,.  \nonumber
\end{eqnarray}

Let $\tilde{\beta}(\xi)$ be the drift of the process $\tilde{X}^N$ when it is in state $\xi$, so that
$$
\tilde{\beta}(\xi) := \sum _{\xi'\neq \xi} q(\xi,\xi')\ (\xi'-\xi)\,,
$$
where $q(\xi,\xi')$ is the jump rate from $\xi$ to $\xi'$. Taking $f=id$ in \eqref{gen} leads to the following formula for the drift
\begin{equation}
\label{eq2}
\left\{
\begin{array}{lcl}
\tilde{\beta}_j(\xi) := 0 & \mbox{ if } & 1 \le j \le \ell_\xi-1\,, \\
& & \\
\displaystyle{\tilde{\beta}_{\ell_\xi}(\xi) := -\sum_{j=\ell_\xi+1}^\infty K(\ell_\xi,j)\ \xi_j - 2\ K(\ell_\xi,\ell_\xi)\ \xi_{\ell_\xi}+\frac{2}{N}\ K(\ell_\xi,\ell_\xi)}\,, & & \\
& & \\
\tilde{\beta}_j(\xi) := K(\ell_\xi,j-\ell_\xi)\ \xi_{j-\ell_\xi} - K(\ell_\xi,j)\ \xi_j & \mbox{ if } &  j \ge \ell_\xi+1, \;j\neq 2\ell_\xi\,,\\
& & \\
\displaystyle{\tilde{\beta}_{2\ell_\xi}(\xi) := K(\ell_\xi,\ell_\xi)\ \left( \xi_{\ell_\xi}-\frac{1}{N} \right)  - K(\ell_\xi,2\ell_\xi)\ \xi_{2\ell_\xi}\,.} & &
\end{array}
\right.
\end{equation}
We also define
\begin{equation}\label{defalpha}
\tilde{\alpha}(\xi) := \sum_{\xi'\neq \xi}q(\xi,\xi')\ \|\xi'-\xi\|_2^2 = \sum_{j=1}^\infty \sum_{\xi'\ne \xi} q(\xi,\xi')\ \left| \xi_j' - \xi_j \right|^2\,.
\end{equation}
It can be written in the form
$$
\tilde{\alpha}(\xi)=\sum_{j=1}^{\infty} \tilde{\alpha}_j(\xi)\,,
$$
where $\tilde{\alpha}_j$ is obtained by taking $f(\xi)=\xi_j \mathbf{e}_j$ in \eqref{QN}, so that
\begin{equation}
\label{eq2vq}
\left\{
\begin{array}{lcl}
\tilde{\alpha}_j(\xi) := 0 & \mbox{ if } & 1 \le j \le \ell_\xi-1\,, \\
& & \\
\displaystyle{ \tilde{\alpha}_{\ell_\xi}(\xi) := \frac{1}{N}\sum_{j=\ell_\xi+1}^\infty K(\ell_\xi,j)\ \xi_j + \frac{4}{N}\ K(\ell_\xi,\ell_\xi)\ \xi_{\ell_\xi} - \frac{4}{N^2}\ K(\ell_\xi,\ell_\xi)}\,, & & \\
& & \\
\displaystyle{ \tilde{\alpha}_j(\xi) := \frac{1}{N}\ K(\ell_\xi,j-\ell_\xi)\ \xi_{j-\ell_\xi} + \frac{1}{N}\ K(\ell_\xi,j)\ \xi_j} & \mbox{ if } &  j \ge \ell_\xi+1, \;j\neq 2\ell_\xi,\\
& & \\
\displaystyle{ \tilde{\alpha}_{2\ell_\xi}(\xi) := \frac{1}{N}\ K(\ell_\xi,\ell_\xi)\ \left( \xi_{\ell_\xi} -\frac{1}{N} \right) + \frac{1}{N}\ K(\ell_\xi,2\ell_\xi)\ \xi_{2\ell_\xi} \,.} & &
\end{array}
\right.
\end{equation}
\subsection{Main results}

For $p\in [1,\infty)$, let $\ell^p$ be the Banach space of $p$-summable real-valued sequences
$$
\ell^p:=\left\{ x=(x_i)_{i\ge 1}\ : \ \|x\|_p := \left( \sum_{i=1}^\infty |x_i|^p \right)^{1/p} < \infty \right\}\,.
$$
We next define the space $\mathcal{X}_{1,1}$ of real-valued sequences with finite first moment by
\beqn
\label{spirou}
\mathcal{X}_{1,1}:= \left\{ x=(x_i)_{i\ge 1} \ : \ \|x\|_{1,1}:=\sum_{i=1}^\infty i\ |x_i| < \infty \right\}\,,
\eeqn
which is a Banach space for the norm $\|.\|_{1,1}$, and its positive cone
$$
\mathcal{X}_{1,1}^+:= \left\{ x=(x_i)_{i\ge 1}\in \mathcal{X}_{1,1} \ : \ x_i\ge 0 \;\;\mbox{ for }\;\; i\ge 1 \right\}\,.
$$
For $m\ge 2$, let $\mathcal{X}_{1,m}$ be the subspace of $\mathcal{X}_{1,1}$ of sequences having their $m-1$ first components equal to zero, namely
\beqn
\label{fantasio}
\mathcal{X}_{1,m} := \left\{ x=(x_i)_{i\ge 1}\in \mathcal{X}_{1,1} \ : \ x_i= 0 \;\;\mbox{ for }\;\; i\in\{ 1, \ldots, m-1\} \right\}\,,
\eeqn
and $\mathcal{X}_{1,m}^+:=\mathcal{X}_{1,m}\cap \mathcal{X}_{1,1}^+$.

We assume that there is $\kappa>0$ such that
\beqn
\label{c1}
0 \le K(i,j) = K(j,i) \le \kappa\ i\ j\,, \quad i,j\ge 1\,, \;\;\mbox{ and }\;\; \delta_i:=\inf_{j\ge i}{\{K(i,j)\}}> 0 \;\;\mbox{ for }\;\; i\ge 1\,.
\eeqn
Next, for $i\ge 1$, we define the function $b^{(i)}=\left( b_j^{(i)} \right)_{j\ge 1}$ on $\mathcal{X}_{1,1}$ by
\beqn
\label{spip}
\left\{
\begin{array}{lcl}
b_j^{(i)}(x) := 0 & \mbox{ if } & 1 \le j \le i-1\,, \\
b_i^{(i)}(x) := \displaystyle{- 2\ K(i,i)\ x_i - \sum_{j=i+1}^\infty K(i,j)\ x_j}\,, & & \\
b_j^{(i)}(x) := K(j-i,i)\ x_{j-i} - K(i,j)\ x_j & \mbox{ if } &  j \ge i+1\,.
\end{array}
\right.
\eeqn
Let us point out here that $b^{(i)}(x)$ is closely related to the drift $\tilde{\beta}(x)$ defined by \eqref{eq2} for $x\in\mathcal{X}_{1,i}$.

Consider an initial condition $x_0=(x_{i,0})_{i\ge 1}$ such that
\beqn
\label{c3}
x_0\in \mathcal{X}_{1,1}^+ \;\;\mbox{ with }\;\; x_{1,0}>0 \;\;\;\mbox{ and }\;\;\; \|x_0\|_{1,1}=1\,.
\eeqn

\begin{thm}\label{thc1}
Assume that the coagulation kernel $K$ and the initial condition $x_0$ satisfy \eqref{c1} and \eqref{c3}, respectively.
There is a unique pair of functions $(\ell,x)$ fulfilling the following properties:\\
\noindent (i) there is an increasing sequence of times $(t_i)_{i\ge 0}$ with $t_0=0$ such that
$$
\ell(t):=i \;\;\mbox{ for }\;\; t\in [t_{i-1},t_i) \;\;\mbox{ and }\;\; i\ge 1\,.
$$
 We define
\beqn
\label{c4}
t_\infty := \sup_{i\ge 0}{t_i} = \lim_{i\to\infty} t_i \in (0,\infty]\,.
\eeqn
\noindent (ii) $x=(x_i)_{i\ge 1}\in \mathcal{C}([0,t_\infty);\mathcal{X}_{1,1})$ satisfies $x(0)=x_0$,
\beqn
\label{c4b}
x(t)\in \mathcal{X}_{1,\ell(t)}^+\setminus\mathcal{X}_{1,\ell(t)+1} \;\;\mbox{ for }\;\; t\in [0,t_\infty)\,,
\eeqn
and solves
\begin{equation}
\frac{dx}{dt}(t)  = b^{(\ell(t))}(x(t)) \;\;\mbox{ for }\;\; t\in [0,t_\infty)\setminus\{t_i\ :\ i\ge 0 \}\,.\label{c5a}
\end{equation}
In addition,
\beqn
\label{c6}
x_j(t)>0 \;\;\mbox{ for }\;\; t\in (t_{i-1},t_i] \;\;\mbox{ and }\;\; j\ge i+1
\eeqn
and
\beqn
\label{c7}
\|x(t)\|_{1,1} = \|x_0\|_{1,1} = 1 \;\;\mbox{ for }\;\; t\in [0,t_\infty)\,.
\eeqn
\end{thm}

In other words, for each $i\ge 1$, $x(t)\in \mathcal{X}_{1,i}^+$ and $x_i(t)>0$ for $t\in [t_{i-1},t_i)$ and $dx(t)/dt=b^{(i)}(x(t))$ for $t\in (t_{i-1},t_i)$. Given $t\in [0,t_\infty)$, Theorem~\ref{thc1} asserts that $x(t)\in \mathcal{X}_{1,\ell(t)}^+$ with $x_{\ell(t)}(t)>0$, so that $\ell(t)$ is the minimal size of the particles at time $t$.

\begin{rem}\label{remc1}
The assumption $\|x_0\|_{1,1}=1$ is actually not restrictive: indeed, given $\bar{x}_0\in \Xun^+$ such that $\bar{x}_{1,0}>0$, the initial condition $x_0=\bar{x}_0/\|\bar{x}_0\|_{1,1}$ fulfils \eqref{c3}. If $x$ denotes the corresponding solution to \eqref{c5a} with minimal size $\ell$ and $\bar{x}:=\|\bar{x}_0\|_{1,1} x$, it is straightforward to check that the pair $(\ell, \bar{x})$ satisfies all the requirements of Theorem~\ref{thc1} except \eqref{c7} which has to be replaced by $\|\bar{x}(t)\|_{1,1} = \|\bar{x}_0\|_{1,1}$ for $t\in [0,t_\infty)$.
\end{rem}

We now turn to the connection between the deterministic and stochastic models and establish the following convergence result.

\begin{thm} \label{thc3}
Let $K$ and $x_0$ be a coagulation kernel and a deterministic initial condition satisfying \eqref{c1} and \eqref{c3}, respectively. Consider a sequence  $(X_0^N)_{N\ge 1}$ of stochastic initial configurations in $\ell_\NN^1$ satisfying \eqref{asterix} which are close to $x_0$ in the following sense:
\begin{equation} \label{hyp1}
\P\left( \left\|\frac{1}{N}\ X_0^N -x_0\right\|_1> \frac{1}{N^{1/4}} \right) \le \frac{1}{N^{1/4}}\,.
\end{equation}
Assume further that, for any $i\ge 0$, there is $\kappa_i>0$ such that
\beqn
\label{cubitus}
K(i,j) \le \kappa_i \,, \qquad j\ge i\,, \;\;\;\mbox{ and }\;\;\; \kappa_\infty := \sup{\left\{ \frac{\kappa_i}{i} \right\}} < \infty\,.
\eeqn
Let $x$ be the corresponding solution to \eqref{c5a} with maximal existence time $t_\infty$ defined by \eqref{c4} and, for $N\ge 1$, $X^N$ the Markov process starting from $X_0^N$ defined in Section~\ref{secdescrip}. Then
for all $t\in (0,t_\infty)$ there exist constants $C(t), D(t)>0$ such that for $N$ large enough :
$$ \P\left(\sup_{0\leq s\leq t}    \left\| \frac{1}{N}\ X^N(s)-x(s)  \right\|_1\geq \frac{D(t)}{N^{1/4}} \right) \leq \frac{C(t)}{N^{1/4}}. $$
\end{thm}

We next turn to the life span of the deterministic and stochastic min-driven coagulation models and investigate the possible values of $t_\infty$ as well as the behaviour of the time $T^{X_0}$ after which the stochastic min-driven coagulation process $X$ starting from $X_0\in\ell_\NN^1$ ($\ell_\NN^1$ being defined in \eqref{idefix}) no longer evolves, that is,
\begin{equation}
\label{te}
T^{X_0}:=\inf{\left\{ t\ge 0 \ : \ \| X(t)\|_1=1 \right\}}\,.
\end{equation}
We first establish that, according to the growth of the coagulation kernel $K$, $t_\infty$ is finite or infinite. Note that, in the former case, this means that the minimal size $\ell$ blows up in finite time.

\begin{thm}\label{thc2} Consider an initial condition $x_0$ satisfying \eqref{c3} and let $x$ be the corresponding solution to the min-driven coagulation equations given in Theorem~\ref{thc1} defined on $[0,t_\infty)$, $t_\infty$ being defined in \eqref{c4}.\\
\noindent (i) If $K(i,j)\le \left( \ln{(i+1)} \wedge \ln{(j+1)} \right)/(4A_0)$ for $i,j\ge 1$ and some $A_0>0$ then $t_\infty=\infty$.\\
\noindent (ii) If $K(i,j)\ge a_0\ \left( \ln(i+1) \wedge \ln(j+1) \right)^{1+\alpha}$ for $i,j\ge 1$ and some $a_0>0$ and $\alpha>0$, then $t_\infty<\infty$.
\end{thm}

A more precise result is available for the stochastic min-driven coagulation process under a stronger structural assumption on the coagulation kernel.

\begin{thm}\label{thc4} Assume that the coagulation kernel $K$ is of the form
\begin{equation}
\label{Kmin}
K(i,j)=\phi(i)\wedge\phi(j) \;\;\;\mbox{ where }\;\;\; \phi \;\;\;\mbox{ is a positive increasing function.}
\end{equation}
Then
$$
\sup_{X_0\in\ell_\NN^1} \E(T^{X_0}) < \infty \;\;\;\mbox{  if and only if }\;\;\; \sum_{i=1}^\infty\frac{1}{i\phi(i)}<\infty\,,
$$
the space $\ell_\NN^1$ being defined in \eqref{idefix}.
\end{thm}

The above two results provide conditions on the coagulation kernel $K$ which guarantee that, in a finite time, some mass escapes to infinity, 
or forms a giant particle, of the order of the system. This is the behaviour known as {\em gelation} for the Smoluchowski coagulation equation and the  
Marcus-Lushnikov process, and is known to occur when the coagulation kernel $K$ satisfies $K(i,j)\ge c (ij)^{\lambda/2}$ for some 
$\lambda\in (1,2]$ \cite{EMP02,Je98}. We observe that the growth required on the coagulation kernel is much weaker for the min-driven coagulation models.
In fact the behaviour we have shown is more extreme than gelation, in that all the mass goes to infinity or joins the giant particle. A similar
phenomenon has been called as {\em complete gelation} in the context of the Marcus-Lushnikov process, and is known to occur instantaneously,
as $N\to\infty$, whenever $K(i,j)\ge ij(\log(i+1)\log(j+1))^\alpha$ and $\alpha>1$ \cite{MR1722986}.

\section{The deterministic min-driven coagulation equation}\label{tdmde}
\setcounter{thm}{0}
\setcounter{equation}{0}
In this section, we investigate the well-posedness of the min-driven coagulation equation \eqref{c5a}. It is clearly an infinite system of ordinary differential equations which is linear on the time intervals where the minimal size $\ell$ is constant. We will thus first study the well-posedness for this reduced system, assuming the coefficients to be bounded in a first step to be able to apply the Cauchy-Lipschitz theorem and relaxing this assumption afterwards by a compactness method. We also pay attention to the first vanishing time of the first component which was initially positive. The proof of Theorem~\ref{thc1} is then performed by an induction argument.

\subsection{An auxiliary infinite system of differential equations}\label{aaisode}

Consider $i\ge 1$ and a sequence $(a_j)_{j\ge 1}$ of real numbers satisfying
\beqn
\label{b1}
0 < a_j \le A\ j\,, \quad j\ge 1\,,
\eeqn
for some $A>0$. We define the function $F=(F_j)_{j\ge 1}$ on $\Xun$ by
\beqn
\label{b2}
\left\{
\begin{array}{lcl}
F_j(y) := 0 & \mbox{ if } & 1\le j \le i-1\,,\\
F_i(y) := \displaystyle{-a_i\ y_i - \sum_{j=i}^\infty a_j\ y_j}\,, & & \\
F_j(y):= a_{j-i}\ y_{j-i} - a_j\ y_j & \mbox{ if } & j\ge i+1
\end{array}
\right.
\eeqn
for $y\in \Xun$. Note that \eqref{b1} ensures that $F(y)\in\ell^1$ for $y\in\Xun$ and that $F(y)\in\Xuni$.

\begin{prop}\label{prb1}
Consider a sequence $(a_j)_{j\ge 1}$ satisfying \eqref{b1} and an initial condition $y_0=(y_{j,0})_{j\ge 1}\in \Xuni$. There is a unique solution $y\in\mathcal{C}([0,\infty);\Xuni)$ to the Cauchy problem
\beqn
\label{b3}
\frac{dy}{dt} = F(y)\,, \quad y(0)=y_0\,.
\eeqn
Moreover, for each $t>0$, $y$ and $dy/dt$ belong to $L^\infty(0,t;\Xuni)$ and $L^\infty(0,t;\ell^1)$, respectively, and
\beqn
\label{evian}
\sum_{j=i}^\infty j\ y_j(t) = \sum_{j=i}^\infty j\ y_{j,0}\,.
\eeqn
\end{prop}

We first consider the case of a bounded sequence $(a_j)_{j\ge 1}$.

\begin{lem}\label{leb2}
Consider a sequence $(a_j)_{j\ge 1}$ satisfying
\beqn
\label{b5}
0 < a_j \le A_0\,, \quad j\ge 1\,,
\eeqn
for some $A_0>0$ and an initial condition $y_0=(y_{j,0})_{j\ge 1}\in \Xuni$. Then there is a unique solution $y\in\mathcal{C}([0,\infty);\Xuni)$ to the Cauchy problem \eqref{b3} and
\beqn
\label{vichy}
\sum_{j=i}^\infty j\ y_j(t) = \sum_{j=i}^\infty j\ y_{j,0}\,, \quad t\ge 0\,.
\eeqn
\end{lem}

\noindent\textbf{Proof.}  It readily follows from \eqref{b2} and \eqref{b5} that, given $y\in\Xuni$ and $\hat{y}\in \Xuni$, we have
\beqn
\label{b7}
\| F(y)-F(\hat{y})\|_{1,1}\le 4A_0\ \| y-\hat{y}\|_{1,1}\,,
\eeqn
while the first $i-1$ components of $F(y)$ vanish. Therefore, $F$ is a Lipschitz continuous map from $\Xuni$ to $\Xuni$ and the Cauchy-Lipschitz theorem guarantees the existence and uniqueness of a solution $y\in\mathcal{C}([0,\infty);\Xuni)$ to \eqref{b3}. 

Next, let $(g_j)_{j\ge 1}$ is a sequence of real numbers satisfying $0\le g_j\le G\ j$ for $j\ge 1$ and some $G>0$. We deduce from \eqref{b3}, \eqref{b5}, and the summability properties of $y$ that
\beqn
\label{b25}
\frac{d}{dt} \sum_{j=i}^\infty g_j\ y_j(t) = \sum_{j=i}^\infty \left( g_{i+j}-g_i-g_j \right)\ a_j\ y_j(t)\,, \quad t\ge 0\,.
\eeqn
In particular, the choice $g_j=j$, $j\ge 1$, gives \eqref{vichy}. \qed

\medskip

\noindent\textbf{Proof of Proposition~\ref{prb1}.}  For $m\ge 1$ and $j\ge 1$, we put $a_j^m:= a_j\wedge m$. Since the sequence $(a_j^m)_{j\ge 1}$ is bounded, it follows from Lemma~\ref{leb2} that there is a unique solution $y^m=(y_j^m)_{j\ge 1}\in \mathcal{C}([0,\infty);\Xuni)$ to the Cauchy problem
\bear
\frac{dy_i^m}{dt} & = & -a_i^m\ y_i^m - \sum_{j=i}^\infty a_j^m\ y_j^m\,, \label{b8}\\
\frac{dy_j^m}{dt} & = & a_{j-i}^m\ y_{j-i}^m - a_j^m\ y_j^m\,, \quad j\ge i+1\,, \label{b9}
\eear
with initial condition $y^m(0)=y_0$. Introducing $\sigma_j^m:=\mbox{sign}(y_j^m)$, we infer from \eqref{b1}, \eqref{b8}, and \eqref{b9} that
\bean
\frac{d}{dt} \|y^m\|_{1,1} & = & \sum_{j=i}^\infty j\ \sigma_j^m\ \frac{d y_j^m}{dt} \\
& = & - i\ a_i^m\ |y_i^m| - \sum_{j=i}^\infty i\ a_j^m\ \sigma_i^m\ y_j^m + \sum_{j=2i}^\infty j\ a_{j-i}^m\ \sigma_j^m\ y_{j-i}^m - \sum_{j=i+1}^\infty j\ a_j^m\ |y_j^m| \\
& = & \sum_{j=i}^\infty \left( (i+j)\ \sigma_{i+j}^m\ \sigma_j^m - i\ \sigma_i^m\ \sigma_j^m - j \right)\ a_j^m\ |y_j^m| \\
& \le & 2i\ \sum_{j=i}^\infty a_j^m\ |y_j^m| \le 2Ai\ \|y^m\|_{1,1}\,,
\eean
hence
\beqn
\label{b10}
\|y^m(t)\|_{1,1} \le \|y_0\|_{1,1}\ e^{2Ai t}\,, \quad t\ge 0\,.
\eeqn
It next readily follows from \eqref{b1}, \eqref{b8}, and \eqref{b9} that
\bean
\left| \frac{dy_i^m}{dt} \right| & \le & Ai\ |y_i^m| + A\ \|y^m\|_{1,1}\,, \\
\left| \frac{dy_j^m}{dt} \right| & \le & A(j-i)\ |y_{j-i}^m| + Aj\ |y_j^m|\,, \quad j\ge i+1\,,
\eean
and thus
\beqn
\label{b11}
\sum_{j=i}^\infty \left| \frac{dy_j^m}{dt}(t) \right| \le 3A\ \|y^m(t)\|_{1,1}\le 3A\ \|y_0\|_{1,1}\ e^{2Ai t}\,, \quad t\ge 0
\eeqn
by \eqref{b10}.

Now, for all $j\ge 1$ and $T>0$, the sequence of functions $(y_j^m)_{N\ge 1}$ is bounded in $W^{1,\infty}(0,T)$ by \eqref{b10} and \eqref{b11} and thus relatively compact in $\mathcal{C}([0,T])$ by the Arzel\`a-Ascoli theorem. Consequently, there are a subsequence $(m_k)_{k\ge 1}$, $m_k\to\infty$, and a sequence of functions $y=(y_j)_{j\ge 1}$ such that
\beqn
\label{b12}
\lim_{k\to\infty} \sup_{t\in [0,T]}{\left| y_j^{m_k}(t) - y_j(t) \right| } = 0 \;\;\mbox{ for }\;\; j\ge 1 \;\;\mbox{ and }\;\; T>0\,.
\eeqn
If $j\ge i+1$, it is straightforward to deduce from \eqref{b9} and \eqref{b12} that $y_j$ actually belongs to $\mathcal{C}^1([0,\infty))$ and solves
\beqn
\label{b13}
\frac{dy_j}{dt} = a_{j-i}\ y_{j-i} - a_j\ y_j\,, \quad y_j(0)=y_{j,0}\,.
\eeqn
In addition, \eqref{b10} and \eqref{b12} imply that $y(t)\in\Xuni$ for all $t\ge 0$ and satisfies
\beqn
\label{b13b}
\|y(t)\|_{1,1} \le \|y_0\|_{1,1}\ e^{2Ai t}\,, \quad t\ge 0\,.
\eeqn

Passing to the limit in \eqref{b8} is more difficult because of the infinite series in its right-hand side. For that purpose, we need an additional estimate to control the tail of the series which we derive now: we first recall that, since $y_0\in\Xun$, a refined version of the de la Vall\'ee-Poussin theorem ensures that there is a nonnegative and non-decreasing convex function $\zeta\in \mathcal{C}^\infty([0,\infty))$ such that $\zeta(0)=0$, $\zeta'$ is a concave function,
\beqn
\label{b14}
\lim_{r\to\infty} \frac{\zeta(r)}{r} = \infty\,, \;\;\mbox{ and }\;\; \sum_{j=i}^\infty \zeta(j)\ |y_{j,0}| < \infty\,,
\eeqn
see \cite{DM75,Le77}. We infer from \eqref{b1}, \eqref{b8}, \eqref{b9}, and the properties of $\zeta$ that
\bean
\frac{d}{dt} \sum_{j=i}^\infty \zeta(j)\ |y_j^m| & = & \sum_{j=i}^\infty \left( \zeta(i+j)\ \mbox{sign}(y_{i+j}^m)\ \mbox{sign}(y_j^m) - \zeta(i)\ \mbox{sign}(y_{i}^m)\ \mbox{sign}(y_j^m) - \zeta(j) \right)\ a_j^m\ |y_j^m| \\
& \le & \sum_{j=i}^\infty \left( \zeta(i+j) + \zeta(i) - \zeta(j) \right)\ a_j^m\ |y_j^m| \\
& \le & \sum_{j=i}^\infty \left( \int_0^j \int_0^i \zeta''(r+s) \ dsdr + 2\ \zeta(i) \right)\ a_j^m\ |y_j^m| \\
& \le & \sum_{j=i}^\infty \left( \int_0^j i\ \zeta''(r) \ dr + 2\ \zeta(i) \right)\ a_j^m\ |y_j^m| \\
& \le & \sum_{j=i}^\infty \left( i\ \zeta'(j) + 2\ \zeta(i) \right)\ a_j^m\ |y_j^m| \\
& \le & 2A\ \zeta(i)\ \|y^m\|_{1,1} + A i\ \sum_{j=i}^\infty j\ \zeta'(j)\ |y_j^m|\,.
\eean
Owing to the concavity of $\zeta'$, we have $j\ \zeta'(j) \le 2\ \zeta(j)$ for $j\ge 1$ \cite[Lemma~A.1]{La01}. Inserting this estimate in the previous inequality and using \eqref{b10}, we end up with
$$
\frac{d}{dt} \sum_{j=i}^\infty \zeta(j)\ |y_j^m(t)| \le 2Ai\ \sum_{j=i}^\infty \zeta(j)\ |y_j^m(t)| + 2A\ \zeta(i)\ \|y_0\|_{1,1}\ e^{2Ai t}\,, \quad t\ge 0\,,
$$
and thus
\beqn
\label{b15}
\sum_{j=i}^\infty \zeta(j)\ |y_j^m(t)| \le \left( \sum_{j=i}^\infty \zeta(j)\ |y_{j,0}| + 2A\ \zeta(i)\ \|y_0\|_{1,1}\ t \right)\ e^{2Ai t}\,, \qquad t\ge 0\,,
\eeqn
the right-hand side of \eqref{b15} being finite by \eqref{b14}. It first follows from \eqref{b12} and \eqref{b15} by the Fatou lemma that
\beqn
\label{b16}
\sum_{j=i}^\infty \zeta(j)\ |y_j(t)| \le \left( \sum_{j=i}^\infty \zeta(j)\ |y_{j,0}| + 2A\ \zeta(i)\ \|y_0\|_{1,1}\ t \right)\ e^{2Ai t}\,, \qquad t\ge 0\,.
\eeqn
Notice next that, thanks to the superlinearity \eqref{b14} of $\zeta$, the estimates \eqref{b15} and \eqref{b16} provide us with a control of the tail of the series $\sum j\ y_j^m$ and $\sum j\ y_j$ which does not depend on $m$. More precisely, we infer from \eqref{b15}, \eqref{b16}, and the convexity of $\zeta$ that, for $T>0$, $t\in [0,T]$, and $J\ge 2i$,
\bean
\left\| (y^{m_k} - y)(t) \right\|_{1,1} & \le & \sum_{j=i}^{J-1} j\ \left| (y_j^{m_k} -y_j)(t) \right| + \sum_{j=J}^\infty j\ \left( |y_j^{m_k}(t)|+ |y_j(t)| \right) \\
& \le & \sum_{j=i}^{J-1} j\ \left| (y_j^{m_k} -y_j)(t) \right| + \frac{J}{\zeta(J)}\ \sum_{j=J}^\infty \zeta(j)\ \left( |y_j^{m_k}(t)|+ |y_j(t)| \right) \\
& \le & \sum_{j=i}^{J-1} j\ \left| (y_j^{m_k} -y_j)(t) \right| + \frac{2J}{\zeta(J)}\ \left( \sum_{j=i}^\infty \zeta(j)\ |y_{j,0}| + 2A\ \zeta(i)\ \|y_0\|_{1,1}\ T \right)\ e^{2Ai T}\,.
\eean
Owing to \eqref{b12}, we may pass to the limit as $k\to\infty$ in the previous inequality to deduce that
$$
\limsup_{k\to\infty} \sup_{t\in [0,T]} \left\| (y^{m_k} - y)(t) \right\|_{1,1} \le \frac{2J}{\zeta(J)}\ \left( \sum_{j=i}^\infty \zeta(j)\ |y_{j,0}| + 2A\ \zeta(i)\ \|y_0\|_{1,1}\ T \right)\ e^{2Ai T}\,.
$$
We next use \eqref{b14} to let $J\to\infty$ in the previous inequality and conclude that
\beqn
\label{b17}
\lim_{k\to\infty} \sup_{t\in [0,T]} \left\| (y^{m_k} - y)(t) \right\|_{1,1}=0\,.
\eeqn
Recalling \eqref{b1}, it is straightforward to deduce from \eqref{b17} that
$$
\lim_{k\to\infty} \sup_{t\in [0,T]} \left| \sum_{j=i}^\infty a_j^{m_k}\ y_j^{m_k}(t) - \sum_{j=i}^\infty a_j\ y_j(t) \right|=0
$$
for all $T>0$, from which we conclude that $y_i$ belongs to $\mathcal{C}^1([0,\infty))$ and solves
\beqn
\label{b18}
\frac{dy_i}{dt} = -a_i\ y_j - \sum_{j=i}^\infty a_j\ y_j\,, \quad y_i(0)=y_{i,0}\,.
\eeqn
Another consequence of \eqref{b17} is that $y\in \mathcal{C}([0,\infty);\Xuni)$ and is thus locally bounded in $\Xun$. This property in turn provides the boundedness of $dy/dt$ in $\ell^1$, the proof being similar to that of \eqref{b11}. We finally use once more \eqref{b17} to deduce from \eqref{vichy} (satisfied by $y^{m_k}$ thanks to Lemma~\ref{leb2}) that \eqref{evian} holds true. We have thus established the existence part of Proposition~\ref{prb1}.

As for uniqueness, if $y$ and $\hat{y}$ are two solutions to the Cauchy problem \eqref{b3}, a computation similar to that leading to \eqref{b10} gives $\|y(t)-\hat{y}(t)\|_{1,1} \le \|y(0)-\hat{y}(0)\|_{1,1}\ e^{2Ai t}=0$ for $t\ge 0$. Consequently, $y=\hat{y}$ and the uniqueness assertion of Proposition~\ref{prb1} is proved. \qed

\begin{rem}\label{reb2b}
In fact, the derivation of \eqref{b15} is formal as the series $\sum \zeta(j) y_j^m$ is not known to converge \textit{a priori} (recall that $\zeta(j)$ is superlinear by \eqref{b14}). It can be justified rigorously by using classical truncation arguments. More specifically, for $R\ge 1$, define $\zeta_R(r)=\zeta(r)$ for $r\in [0,R]$ and $\zeta_R(r)=\zeta(R)+\zeta'(R) (r-R)$ for $r\ge R$. Then $\zeta_R$ enjoys the same properties as $\zeta$ and the sequence $(\zeta_R(j))_{j\ge 1}$ grows linearly with respect to $j$. We can then use \eqref{b25} to perform a similar computation as the one above leading to \eqref{b15} and obtain a bound on $\sum \zeta_R(j)\ y_j^m$ which does not depend on $R$ neither on $m$. The expected result then follows by letting $R\to\infty$ with the help of the Fatou lemma.
\end{rem}

We now turn to specific properties of solutions to \eqref{b3} when $y_0\in \Xuni^+$.

\begin{prop}\label{prb3}
Consider a sequence $(a_j)_{j\ge 1}$ satisfying \eqref{b1}, an initial condition $y_0=(y_{j,0})_{j\ge 1}\in \Xuni$ such that
\beqn
\label{b19}
y_0\in \Xuni^+ \;\;\mbox{ and }\;\; y_{i,0}>0\,,
\eeqn
and let $y$ be the corresponding solution to the Cauchy problem \eqref{b3}. There are $t_*\in (0,\infty]$ and $t_{*,1}\in [t_*,\infty]$ such that
\bear
& & y_i(t)>0 \;\;\mbox{ for }\;\; t\in [0,t_*) \;\;\mbox{ and }\;\; y_i(t_*)=0\,, \label{b20}\\
& & y_{ki}(t)>0 \;\;\mbox{ for }\;\; t\in (0,t_*) \;\;\mbox{ and }\;\; k\ge 2\,, \label{b21}\\
& & y_j(t)\ge 0 \;\;\mbox{ for }\;\; t\in [0,t_*) \;\;\mbox{ and }\;\; j\ge i+1\,, \label{b22}\\
& & y_j(t)>0 \;\;\mbox{ for }\;\; t\in [0,t_*) \;\;\mbox{ if }\;\; j\ge i+1 \;\;\mbox{ and }\;\; y_{j,0}>0\,, \label{b23}\\
& & \frac{dy_i}{dt}(t)<0 \;\;\mbox{ for }\;\; t\in [0,t_{*,1})\,, \label{b23b}
\eear
and
\beqn
\label{b24}
\|y(t)\|_{1,1} = \|y_0\|_{1,1} \;\;\mbox{ for }\;\; t\in [0,t_*)\,.
\eeqn
If $t_*<\infty$, then $t_{*,1}>t_*$ and the properties \eqref{b21}, \eqref{b22}, \eqref{b23}, and \eqref{b24} also hold true for $t=t_*$.
\end{prop}

\noindent\textbf{Proof.}
We define
$$
t_*:= \sup\left\{ t>0\ : \ y_i(s)>0 \;\;\mbox{ for }\;\; s \in [0,t) \right\}\,,
$$
and first notice that $t_*>0$ due to the continuity of $y_i$ and the positivity \eqref{b19} of $y_{i,0}$. Clearly, $y_i$ fulfils \eqref{b20}.

Consider next $j\in \{i+1, \ldots, 2i-1\}$ (if this set is non-empty). Since $y(t)\in\Xuni$ for $t\ge 0$, it follows from \eqref{b3} that, for $t\in [0,t_*)$, $dy_j(t)/dt = - a_j\ y_j(t)$ and thus $y_j(t)=y_{j,0}\ e^{-a_j t}\ge 0$. We next deduce from \eqref{b3} that, for $t\in [0,t_*)$, $dy_{2i}(t)/dt=a_i\ y_i(t)-a_{2i}\ y_{2i}(t)\ge -a_{2i}\ y_{2i}(t)$, whence $y_{2i}(t)\ge y_{2i,0}\ e^{-a_{2i} t}\ge 0$. We next argue in a similar way to prove by induction that $y_j(t)\ge 0$ for $t\in [0,t_*)$ so that $y$ fulfils \eqref{b22}.

We now improve the positivity properties of $y$ and prove \eqref{b21} and \eqref{b23}. Consider first $j\ge i+1$ for which $y_{j,0}>0$. By \eqref{b3} and \eqref{b22}, we have $dy_j(t)/dt=a_{j-i}\ y_{j-i}(t)-a_j\ y_j(t)\ge -a_j\ y_j(t)$ for $t\in [0,t_*)$, whence $y_j(t)\ge y_{j,0}\ e^{-a_j t}> 0$ and \eqref{b23}. To prove \eqref{b21}, we argue by contradiction and assume that there are $k\ge 2$ and $t_0\in (0,t_*)$ (or $t_0\in (0,t_*]$ if $t_*<\infty$) such that $y_{ki}(t_0)=0$. We infer from \eqref{b3} and the variation of constants formula that
$$
0=y_{ki}(t_0)=e^{-a_{ki} t_0}\ y_{ki,0} + a_{(k-1)i}\ \int_0^{t_0} e^{-a_{ki}(t_0-s)}\ y_{(k-1)i}(s)\ ds\,.
$$
The non-negativity of $y_{ki,0}$ and $y_{(k-1)i}$ and the continuity of $y_{(k-1)i}$ then imply that $y_{ki,0}=0$ and $y_{(k-1)i}(t)=0$ for $t\in [0,t_0]$. At this point, either $k=2$ and we have a contradiction with \eqref{b20}. Or $k>2$ and we proceed by induction to show that $y_{li}(t)=0$ for $t\in [0,t_0]$ and $l\in \{1,\ldots,k\}$, again leading us to a contradiction with \eqref{b20}.

The property \eqref{b23b} now follows from \eqref{b1} and \eqref{b21}: indeed, by \eqref{b3} we have
$$
\frac{dy_i}{dt}(t) = - a_i\ y_i(t) - \sum_{j=i}^\infty a_j\ y_j(t) \le -a_{2i}\ y_{2i}(t)<0
$$
for $t\in [0,t_*)$ (and also for $t=t_*$ if $t_*<\infty$,) so that
$$
t_{*,1} := \sup\left\{ t>0\ : \ \frac{dy_i}{dt}(s)<0 \;\;\mbox{ for }\;\; s \in [0,t) \right\}\ge t_*\,,
$$
and $t_{*,1}>t_*$ if $t_*<\infty$.

Finally, since $y(t)$ belongs to $\mathcal{X}_{1,i}^+$ for $t\in [0,t_*)$, \eqref{b24} readily follows from \eqref{evian}. \qed

\medskip

We next turn to the study of the finiteness of the time $t_*$ defined in Proposition~\ref{prb3}.

\begin{prop}\label{prb5}
Consider a sequence $(a_j)_{j\ge 1}$ satisfying \eqref{b1}, an initial condition $y_0=(y_{j,0})_{j\ge 1}\in \Xuni$ satisfying \eqref{b19} and let $y$ be the corresponding solution to the Cauchy problem \eqref{b3}. Assume further that there is $\delta_0>0$ such that
\beqn
\label{b26}
0<\delta_0 \le a_j\,, \quad j\ge 1\,.
\eeqn
If  $t_*\in (0,\infty]$ denotes the time introduced in Proposition~\ref{prb3}, then $t_*\in (0,\infty)$.
\end{prop}

\noindent\textbf{Proof.} For $t\ge 0$, we put
$$
M_0(t):= \sum_{j=i}^\infty y_j(t) \;\;\mbox{ and }\;\; M_{-1}(t):= \sum_{j=i}^\infty \frac{y_j(t)}{j}\,.
$$
By \eqref{b20}, $M_0(t)>0$ for $t\in [0,t_*)$ and it follows from \eqref{b25} that
\bean
\frac{d}{dt} \left( \frac{M_{-1}}{M_0} \right) & = & \frac{1}{M_0}\ \sum_{j=i}^\infty \left( \frac{1}{i+j} - \frac{1}{i} - \frac{1}{j} \right)\ a_j\ y_j + \frac{M_{-1}}{M_0^2}\ \sum_{j=i}^\infty a_j\ y_j \\
& = & \frac{1}{M_0}\ \sum_{j=i}^\infty \left( \frac{1}{i+j} - \frac{1}{j} + \frac{M_{-1}}{M_0}- \frac{1}{i} \right)\ a_j\ y_j\,.
\eean
Observing that
$$
\frac{1}{i+j} \le \frac{1}{j} \;\;\mbox{ and }\;\;  \frac{M_{-1}}{M_0} \le \frac{1}{i}\,,
$$
we infer from \eqref{b26} that
\bean
\frac{d}{dt} \left( \frac{M_{-1}}{M_0} \right) & \le & \frac{\delta_0}{M_0}\ \sum_{j=i}^\infty \left( \frac{1}{i+j} - \frac{1}{j} + \frac{M_{-1}}{M_0}- \frac{1}{i} \right)\ y_j\\
& \le & \frac{\delta_0}{M_0}\ \left( \sum_{j=i}^\infty \left( \frac{1}{i+j} - \frac{1}{i} \right)\ y_j -M_{-1} + \frac{M_{-1}}{M_0}\ M_0 \right) \\
& \le & - \frac{\delta_0}{M_0}\ \sum_{j=i}^\infty \frac{j}{i(i+j)}\ y_j \le - \frac{\delta_0}{2i\ M_0}\ \sum_{j=i}^\infty  y_j \le - \frac{\delta_0}{2i}\,.
\eean
Consequently, we have
$$
0\le \frac{M_{-1}}{M_0}(t) \le \frac{M_{-1}}{M_0}(0) - \frac{\delta_0}{2i}\ t
$$
for $t\in [0,t_*)$ which implies that $t_*\le (2i M_{-1}(0))/(\delta_0 M_0(0))\le 2/\delta_0$ and is thus finite. \qed

\subsection{Proof of Theorem~\ref{thc1}}\label{pot}

The construction of the functions $(\ell,x)$ is performed by induction on the minimal size, noticing that $x$ solves an infinite system of ordinary differential equations similar to \eqref{b3} on each time interval where $\ell$ is constant.

\medskip

\noindent\textbf{Proof of Theorem~\ref{thc1}.} \\
\noindent\textbf{Step~1:} By \eqref{c1}, the sequence $(K(1,j))_{j\ge 1}$ fulfils the assumptions \eqref{b1} (with $A=\kappa$) and \eqref{b26} (with $\delta_0=\delta_1$) while $x_0$ satisfies \eqref{b19} with $i=1$. According to Propositions~\ref{prb1}, \ref{prb3}, and~\ref{prb5}, there is a unique solution $x^{(1)}\in\mathcal{C}([0,\infty);\Xun)$ to the Cauchy problem
$$
\frac{dx^{(1)}}{dt} = b^{(1)}(x^{(1)})\,, \qquad x^{(1)}(0)=x_0\,,
$$
and there is $t_1\in (0,\infty)$ such that
\bean
& & x_1^{(1)}(t)>0 \;\;\mbox{ for }\;\; t\in [0,t_1) \;\;\mbox{ and }\;\; x_1^{(1)}(t_1)=0\,, \\
& & x_j^{(1)}(t)>0 \;\;\mbox{ for }\;\; t\in (0,t_1] \;\;\mbox{ and }\;\; j\ge 2\,, \\
& & \left\| x^{(1)}(t) \right\|_{1,1} = \|x_0\|_{1,1} \;\;\mbox{ for }\;\; t\in [0,t_1]\,.
\eean
We then put
$$
\ell(t):= 1 \;\;\mbox{ and }\;\; x(t):=x^{(1)}(t) \;\;\mbox{ for }\;\; t\in [0,t_1)\,.
$$
Clearly, $x$ fulfils \eqref{c4b}, \eqref{c5a}, and \eqref{c7} for $i=1$.

\smallskip

\noindent\textbf{Step~2:} Assume now that we have constructed $(\ell,x)$ up to some time $t_i$ for some $i\ge 1$. On the one hand, owing to \eqref{c1}, the sequence $(K(i+1,j))_{j\ge 1}$ fulfils the assumptions \eqref{b1} (with $A=\kappa\ (i+1)$) and \eqref{b26} (with $\delta_0=\delta_{i+1}$). On the other hand, the sequence $x(t_i)$ belongs to $\mathcal{X}_{1,i+1}^+$ with $x_j(t_i)>0$ for $j\ge i+1$ by \eqref{c6}. We are then in a position to apply Propositions~\ref{prb1}, \ref{prb3}, and~\ref{prb5} and conclude that there is a unique solution $x^{(i+1)}\in\mathcal{C}([t_i,\infty);\mathcal{X}_{1,i+1})$ to the Cauchy problem
$$
\frac{dx^{(i+1)}}{dt} = b^{(i+1)}(x^{(i+1)})\,, \qquad x^{(i+1)}(t_i)=x(t_i)\,,
$$
and there is $t_{i+1}\in (0,\infty)$ such that
\bean
& & x_{i+1}^{(i+1)}(t)>0 \;\;\mbox{ for }\;\; t\in [t_i,t_{i+1}) \;\;\mbox{ and }\;\; x_{i+1}^{(i+1)}(t_{i+1})=0\,, \\
& & x_j^{(i+1)}(t)>0 \;\;\mbox{ for }\;\; t\in (t_i,t_{i+1}] \;\;\mbox{ and }\;\; j\ge i+2\,,\\
& & \left\| x^{(i+1)}(t) \right\|_{1,1} = \|x(t_i)\|_{1,1} \;\;\mbox{ for }\;\; t\in [t_i,t_{i+1}]\,.
\eean
We then put
$$
 \ell(t):= i+1 \;\;\mbox{ and }\;\; x(t):=x^{(i+1)}(t) \;\;\mbox{ for }\;\; t\in [t_i,t_{i+1})\,.
$$
It is then easy to check that $x\in \mathcal{C}([0,t_{i+1};\Xun)$ and fulfils \eqref{c4b}, \eqref{c5a}, \eqref{c6}, and \eqref{c7} for $j\in\{1,\ldots,i+1\}$. This completes the induction process and the proof of the existence part of Theorem~\ref{thc1}.

\smallskip

\noindent\textbf{Step~3:}  If $(\ell,x)$ and $(\hat{\ell},\hat{x})$ both satisfy the properties listed in Theorem~\ref{thc1}, we deduce from Proposition~\ref{prb1} that $x(t)=\hat{x}(t)$ for $t\in [0,t_1\wedge \hat{t}_1]$. In particular, $x_1$ and $\hat{x}_1$ vanish at the same time $t_1\wedge \hat{t}_1$ which implies that $t_1=\hat{t}_1$. We next argue by induction to conclude that $\ell=\hat{\ell}$ and $x=\hat{x}$. \qed

\section{Convergence of the stochastic process}\label{sec:csp}
\setcounter{thm}{0}
\setcounter{equation}{0}

In this section, we study the stochastic process introduced in Section~\ref{secdescrip} and prove Theorem~\ref{thc3}. The proof is performed along the lines of the general scheme developed in \cite{DN08} with the following main differences: the deterministic system of ordinary differential equations \eqref{c5a} considered herein has its solutions in an infinite-dimensional vector space and changes when the minimal size $\ell$ jumps.

\medskip

Let $K$ be a coagulation kernel satisfying \eqref{cubitus}.
We fix an initial condition $x_0$ satisfying \eqref{c3} and let $x$ be the corresponding solution to \eqref{c5a}. Owing to \eqref{c7} and \eqref{cubitus}, we may argue as in the proof of Proposition~\ref{prb1} to show that, for $i\ge 1$,
\begin{equation}
\label{senechal}
\left \| \frac{dx}{dt}(t) \right\|_1 \le  3\kappa_i \,, \qquad t\in [t_{i-1},t_i]\,.
\end{equation}
Consider a sequence of random initial data $\left( X_0^N \right)_{N\ge 1}$ in $\ell_\NN^1$ satisfying \eqref{asterix} and \eqref{hyp1}. For each $N\ge 1$, $X^N$ denotes the Markov process described in Section~\ref{secdescrip} starting from $X_0^N$ and $\tilde{X}^N:=X^N/N$ its renormalized version. To prove Theorem~\ref{thc3},  we need to introduce some specific times relative to the extinction of some sizes of particle. Let $T_0^N=0$ and define
\begin{equation}
\label{semaphore}
T_i^N:=\inf\{t>T_{i-1}^N \ : \ X_i^N(t)=0\}\,, \quad  \sigma_i^N:=T_i^N-T_{i-1}^N\,, \quad i\ge 1\,.
\end{equation}
We also put $s_i:=t_i-t_{i-1}$ for $i\ge 1$, the times $(t_i)_{i\ge 0}$ being defined in Theorem~\ref{thc1}.

We begin by proving the following proposition.

\begin{prop} \label{prop synthese}
For all $I\geq 0, $ there exist positive constants $C_0(I)$, $C_0(I)'$, and an integer $N_0(I)$ such that
$$ \P\left(\sup_{0\leq t\leq T_I^N}    || \tilde{X}^N(t)-x(t)  ||_1 > \frac{C_0(I)}{N^{1/4}} \right) \leq \frac{C_0(I)'}{N^{1/4}} \;\;\;\mbox{ for }\;\;\; N\ge N_0(I)\,. $$
\end{prop}

Two steps are needed to prove Proposition~\ref{prop synthese}: we first consider $i\ge 1$ and work on the interval  $[T^N_{i-1}, T^N_{i} ]$, showing that the behaviour at any time $t\in (T_{i-1}^N,T_i^N]$ only depends on the behaviour at the ``initial'' time $ T^N_{i-1} $  (Proposition~\ref{prop1}). We then argue by induction on $i$ to prove a ``global'' convergence result (Proposition~\ref{prop2}).

\begin{prop}  \label{prop1}
For all $i\geq 1$ and  $\gamma>0$, there exist positive constants $C_1(\gamma,i)$, $C_1(i)'$, $\bar{s}_i\in (s_i,s_{i}+1)$, $\eta_i$, and an integer $N_1(\gamma,i)$ such that
\begin{equation}
\label{panoramix}
x_i^{(i)}(t_{i-1}+\bar{s}_i)<0 \,, \qquad \frac{dx_i^{(i)}}{dt}(t_{i-1}+s)\le - \eta_i < 0 \;\;\mbox{ for }\;\; s\in [0,\bar{s}_i]\,,
\end{equation}
\begin{eqnarray*}
& & \P\left(\sup_{0\le s\leq \sigma_i^N} \| \tilde{X}^N(T_{i-1}^N+s)-x^{(i)}(t_{i-1}+s)\|_1 > \frac{C_1(\gamma,i)}{N^{1/4} } \right)\leq \frac{C_1(i)'}{N^{1/4}} +\P(\Omega_{i,\gamma}^c)\,, \\
& & \P\left( \sigma_i^N > \bar{s}_i \right) \le \frac{C_1(i)'}{N^{1/4}} +\P(\Omega_{i,\gamma}^c)\,, \\
\end{eqnarray*}
for $N\ge N_1(\gamma,i)$, where
$$
\Omega_{i,\gamma}:=\left\{ ||\tilde{X}^N(T_{i-1}^N) -x(t_{i-1})||_1\leq \frac{\gamma}{N^{1/4}}  \right\}\,,
$$
and $x^{(i)}~: [t_{i-1},\infty) \to \Xun$ denotes the solution to the differential equation
\begin{equation}\label{equadiff}
\frac{d x^{(i)}}{dt}(t)=b^{(i)}(x^{(i)}(t)) \quad \mbox{for } t\ge t_{i-1}\,, \quad x^{(i)}(t_{i-1})=x(t_{i-1})\,.
\end{equation}
\end{prop}

\noindent\textbf{Proof.} Fix $i\ge 1$ and set $\tilde{x}:=x^{(i)}$ to simplify the notation. Recall that $x(t)=x^{(i)}(t)$ for $t\in [t_{i-1},t_i]$. By Section~\ref{secdescrip}, we have for $0\le s\le \sigma_i^N $,
\begin{eqnarray*}
       \tilde{x}(t_{i-1}+s) & = & x(t_{i-1}) + \int_0^s b^{(i)} (\tilde{x}(t_{i-1}+t)) \ dt, \\
       \tilde{X}^N(T^N_{i-1}+s) & = & \tilde{X}^N(T_{i-1}) + \int_0^s \tilde{\beta} ( \tilde{X}^N (T^N_{i-1}+t)) \ dt + M_s^N,
\end{eqnarray*}
where $(M_s^N)_{s\geq 0}$ is a $\F_s^{(i)}$-martingale, $\F_s^{(i)}:=\sigma\left( X_{T^N_{i-1}+t} \ : \ t\in [0,s] \right)$, and $\tilde{\beta}$ is the drift of the process $\tilde{X}^N$ defined in \eqref{eq2}.  Subtracting the above two identities, we obtain
\begin{eqnarray} \label{eq3}
 & & \tilde{X}^N(T^N_{i-1}+s)-\tilde{x}(t_{i-1}+s) \\
 &=& \tilde{X}^N(T^N_{i-1})-x(t_{i-1}) + \int_0^s \left[ \tilde{\beta}( \tilde{X}^N (T^N_{i-1}+t)) -b^{(i)} (\tilde{X}^N(T^N_{i-1}+t)) \right]\ dt \nonumber\\
 & + &  \int_0^s \left[ b^{(i)} ( \tilde{X}^N (T^N_{i-1}+t)) -b^{(i)} (\tilde{x}(t_{i-1}+t)) \right]\ dt + M_s^N.\nonumber
\end{eqnarray}
We now aim at using the representation formula \eqref{eq3} to estimate $\|\tilde{X}^N(T^N_{i-1}+s)-\tilde{x}(t_{i-1}+s)\|_1$ for $s\in [0,\sigma_i^N]$. This requires in particular to estimate the martingale term $M_s^N$ in $\ell^1$. However, a classical way to estimate $M_s^N$ is to use the Doob inequality which gives an $L^2$-bound not suitable for our purposes. To remedy this difficulty, we only use \eqref{eq3} for the first $d$ components of $\tilde{X}^N(T^N_{i-1}+s)-\tilde{x}(t_{i-1}+s)$, the integer $d$ being suitably chosen, and control the tail of the series by the first moment. More precisely, given $d\ge 1$, we introduce the projections $p_d$ and $q_d$ defined in $\ell^1$ by $p_d(y):=(y_1,\ldots,y_d,0,\ldots)$ and $q_d(y)=y-p_d(y)$, $y\in \ell^1$. Clearly,
\begin{equation}\label{z1}
\|p_d(y)\|_1 \le \sqrt{d}\ \|p_d(y)\|_2\,, \qquad y\in \ell^1\,,
\end{equation}
and
\begin{equation}\label{z2}
\|q_d(y)\|_1 \le \frac{\|y\|_{1,1}}{d}\,, \qquad y\in \Xun\,.
\end{equation}
Owing to \eqref{z2} and the boundedness of the first moment of $\tilde{X}^N$ and $\tilde{x}$ (see \eqref{ielosubmarine}, \eqref{c7}, and Lemma~\ref{leb2}), we have for $s\in [0,\sigma_i^N]$
\begin{eqnarray}
& & \|\tilde{X}^N(T^N_{i-1}+s)-\tilde{x}(t_{i-1}+s)\|_1 \nonumber\\
& \le & \|p_d\left( \tilde{X}^N(T^N_{i-1}+s)-\tilde{x}(t_{i-1}+s) \right)\|_1 + \|q_d\left( \tilde{X}^N(T^N_{i-1}+s) \right)\|_1 + \|q_d(\tilde{x}(t_{i-1}+s))\|_1 \nonumber\\
& \le & \|p_d\left( \tilde{X}^N(T^N_{i-1}+s)-\tilde{x}(t_{i-1}+s) \right)\|_1 + \frac{\|\tilde{X}^N(T^N_{i-1}+s)\|_{1,1}}{d} + \frac{\|\tilde{x}(t_{i-1}+s)\|_{1,1}}{d} \nonumber \\
& \le & \|p_d\left( \tilde{X}^N(T^N_{i-1}+s)-\tilde{x}(t_{i-1}+s) \right)\|_1 + \frac{\left( 1 + \|x(t_{i-1})\|_{1,1}\ e^{4\kappa_i s} \right)}{d}  \nonumber \\
& \le & \|p_d\left( \tilde{X}^N(T^N_{i-1}+s)-\tilde{x}(t_{i-1}+s) \right)\|_1 + \frac{\left( 1 + \|x_0\|_{1,1}\ e^{4\kappa_i s} \right)}{d}\,.\label{z3}
\end{eqnarray}

Since $\tilde{\beta}_j -b_j^{(i)}=0$ for all $j\ge 1$ except for $j\in\{i,2i\}$ for which $\tilde{\beta}_i -b_i^{(i)}=2K(i,i)/N$ and $\tilde{\beta}_{2i} -b_{2i}^{(i)}=-K(i,i)/N$ we have
\begin{equation}
\label{z4}
\| \tilde{\beta}(y) -b^{(i)}(y)  ||_1\leq \frac{3K(i,i)}{N} \le \frac{3\kappa_i}{N}\,, \qquad y\in\Xun\,,
\end{equation}
 by \eqref{cubitus}.
 Observing next that $b^{(i)}$ is Lipschitz continuous in $\ell^1$ with Lipschitz constant $3\kappa_i$, we infer from \eqref{eq3}, \eqref{z1}, and \eqref{z4} that
\begin{eqnarray*}
\|p_d \left( \tilde{X}^N  (T_{i-1}^N+s) - \tilde{x}(t_{i-1}+s) \right) \|_1  & \leq & \|p_d\left( \tilde{X}^N(T_{i-1}) - \tilde{x}(t_{i-1}) \right) \|_1 + \frac{3\kappa_i s}{N} \\
 & + &   3\kappa_i \int_0^s   \| \tilde{X}^N (T_{i-1}^N+t) - \tilde{x}(t_{i-1}+t) ||_1\ dt + \sqrt{d}\ \| p_d(M_s^N) \|_2.
\end{eqnarray*}
Combining the above inequality with \eqref{z3} gives
\begin{eqnarray}
\|\tilde{X}^N  (T_{i-1}^N+s) - \tilde{x}(t_{i-1}+s)\|_1  & \leq & \| \tilde{X}^N(T_{i-1}) - \tilde{x}(t_{i-1}) \|_1 + \frac{3\kappa_i s}{N} + \frac{\left( 1 + \|x_0\|_{1,1}\ e^{4\kappa_i s} \right)}{d} \label{z5}\\
 & + &   3\kappa_i \int_0^s   \| \tilde{X}^N (T_{i-1}^N+t) - \tilde{x}(t_{i-1}+t) ||_1\ dt + \sqrt{d}\ \| M_s^N \|_2\,. \nonumber
\end{eqnarray}

At this point, we fix $\bar{s}_i\in (s_i,s_i+1)$ and $\eta_i>0$ such that $\tilde{x}_i(t_{i-1}+\bar{s}_i)<0$ and $d\tilde{x}_i/dt(t_{i-1}+s) < - \eta_i$ for $s\in [0,\bar{s}_i]$ (such a pair $(\bar{s}_i,\eta_i)$ exists as $\tilde{x}_i(t_i)=x_i(t_{i-1}+s_i)=0$ and $d\tilde{x}_i/dt<0$ in $[t_{i-1},t_i]$ by \eqref{b23b}). Let $\gamma>0$ and introduce
$$
\Omega_i' := \left\{ \sup_{s\in [0,\bar{s}_i \wedge \sigma_i^N]} \|M_s^N\|_2 \le \frac{1}{N^{3/8}} \right\}\,.
$$
Choosing an integer $d\in (N^{1/4}, 2 N^{1/4})$, we deduce from \eqref{z5} that, in $\Omega_{i,\gamma}\cap \Omega_i'$, we have for $s\in [0,\bar{s}_i \wedge\sigma_i^N]$
\begin{eqnarray*}
\|\tilde{X}^N(T_{i-1}^N+s) - \tilde{x}(t_{i-1}+s) \|_1  &\leq &  \frac{\gamma}{N^{1/4}}  + \frac{3\kappa_i s}{N} + \frac{\left( 1 + \|x_0\|_{1,1}\ e^{4\kappa_i s} \right)}{N^{1/4}} \\
 & +&   3\kappa_i\ \int_0^s   \| \tilde{X}^N (T_{i-1}^N+t) - \tilde{x}(t_{i-1}+t) ||_1\ dt + \frac{\sqrt{2}}{N^{1/4}} \\
 & \le & \frac{\gamma+C_2}{N^{1/4}}\ e^{4\kappa_i s} + 3\kappa_i\ \int_0^s   \| \tilde{X}^N (T_{i-1}^N+t) - \tilde{x}(t_{i-1}+t) ||_1\ dt
\end{eqnarray*}
for some positive constant $C_2$. After integration, we end up with
\begin{equation}
\label{z6}
\sup_{s\in [0,\bar{s}_i \wedge \sigma_i^N]} \|\tilde{X}^N(T_{i-1}^N+s) - \tilde{x}(t_{i-1}+s) \|_1 \le 5\ \frac{\gamma+C_2}{N^{1/4}}\ \ e^{4\kappa_i \bar{s}_i} \le 5\ \frac{\gamma+C_2}{N^{1/4}}\ \ e^{4\kappa_i (1+s_i)}\,.
\end{equation}
In particular, in $\{\sigma_i^N>\bar{s}_i\}\cap\Omega_{i,\gamma}\cap\Omega_i'$, we have
$$
0 \le \tilde{X}^N_i(T_{i-1}^N+\bar{s}_i) \le \tilde{x}_i(t_{i-1}+\bar{s}_i) + 5\ \frac{\gamma+C_2}{N^{1/4}}\ \ e^{4\kappa_i (1+s_i)} < 0
$$
for $N$ large enough. Consequently, there is $N_1(\gamma,i)$ such that
$$
\Omega_{i,\gamma}\cap\Omega_i' \subset \{\sigma_i^N \le \bar{s}_i \} \;\;\;\mbox{ for }\;\;\; N\ge N_1(\gamma,i)\,.
$$
Recalling \eqref{z6}, we have thus established that, for $N\ge N_1(\gamma,i)$,
\begin{eqnarray}  \P\left(\sup_{s\in [0,\bar{s}_i\wedge\sigma_i^N]} \|\tilde{X}^N (T_{i-1}^N+s)- \tilde{x}(t_{i-1}+s) ||_1  \geq
\frac{C_1(\gamma,i)}{N^{1/4}}\right)& \leq &  \P\left((\Omega_{i,\gamma}\cap\Omega_i')^c\right) \label{z7} \\
& \leq &\P(\Omega_{i,\gamma}^c) + \P\left(\Omega_i'^c\right)\,, \nonumber
\end{eqnarray}
and
\begin{equation}
\label{z8}
\P\left( \sigma_i^N > \bar{s}_i \right) \le \P\left((\Omega_{i,\gamma}\cap\Omega_i')^c\right) \leq \P(\Omega_{i,\gamma}^c) + \P\left(\Omega_i'^c\right)\,,
\end{equation}
with $C_1(\gamma,i) := 5 (\gamma+C_2) e^{4\kappa_i (1+s_i)}$.

To complete the proof, it remains to bound $\P(\Omega_i'^c)$. By the Doob inequality, we have:
$$
\E\left(\sup_{s \in [0,\bar{s}_i \wedge \sigma_i^N]}   \|M_s^N\|_2^2\right) \leq 4  \ \E\left( \|M_{\bar{s}_i \wedge \sigma_i^N}^N\|_2^2\right) \leq 4 \  \E\left( \int_0^{\bar{s}_i \wedge \sigma_i^N}   \tilde{\alpha} \left( \tilde{X}^N(T_{i-1}^N +t) \right)  \ dt \right),
$$
where $\tilde{\alpha}$ is defined by (\ref{defalpha}). According to Section~\ref{secdescrip} and \eqref{cubitus},  it is easy to show that, if $y\in \Xuni$, we have $\tilde{\alpha}(y)\le 5\kappa_i\ \|y\|_1/N$. Since $X^N(s)\in \Xuni$ for $s\in [T_{i-1}^N,T_i^N]$ and $\bar{s}_i<s_i+1$, we conclude that
$$
\E\left(\sup_{s \in [0,\bar{s}_i \wedge \sigma_i^N]}   ||M_s^N||_2^2\right)  \leq \frac{C_3(i)}{N}\,.
$$
Therefore, observing that
$$
\P\left(\Omega_i'^c\right) =\P\left( \sup_{s\in [0,\bar{s}_i \wedge \sigma_i^N]}  || M_s^N||_2^2 > \frac{1}{N^{3/4}} \right),$$
the Markov inequality yields
$$
\P\left(\Omega_i'^c\right)  \leq N^{3/4}\ \E\left(\sup_{s\in [0,\bar{s}_i \wedge \sigma_i^N]}  || M_s^N||_2^2 \right) \leq \frac{C_3(i)}{N^{1/4}}\,.
$$
Proposition~\ref{prop1} then readily follows from \eqref{z7}, \eqref{z8}, and the above bound with $C_1(i)':=C_3(i)$. \qed

\medskip

\begin{prop}\label{prop2}
For all $i\geq 1$, there exist positive constants $a_i$, $b_i$, and an integer $N_2(i)$ such that
\begin{equation}
\label{z9}
\P\left( \|\tilde{X}^N(T_{i-1}^N) -x(t_{i-1})\|_1 > \frac{b_i}{N^{1/4}} \right)\leq  \frac{a_i}{N^{1/4}} \ \text{ for all } N\geq N_2(i)\,.
\end{equation}
\end{prop}

\noindent\textbf{Proof.}
We argue by induction on $i\ge 1$ and first note that \eqref{z9} holds true for $i=1$ with $a_1=b_1=1$ by \eqref{hyp1}. \\
Assume next that \eqref{z9} holds true for some $i\ge 1$. Setting $\tilde{x}:=x^{(i)}$, the function $x^{(i)}$ being defined in Proposition~\ref{prop1}, we have
\begin{equation}
\label{z10}
\|\tilde{X}^N(T_{i}^N) -x(t_{i})\|_1\leq \|\tilde{X}^N(T_{i}^N) -\tilde{x}(t_{i-1}+ \sigma_i^N  )|
\|_1 + \| \tilde{x}(t_{i-1}+\sigma_i^N)-\tilde{x}(t_{i})\|_1\,.
\end{equation}
On the one hand, it follows from \eqref{z9} for $i$ and Proposition~\ref{prop1} with $\gamma=b_i$ that we have
\begin{eqnarray}
\P\left( \|\tilde{X}^N(T_{i}^N) -\tilde{x}(t_{i-1}+ \sigma_i^N  )\|_1 > \frac{C_1(b_i,i)}{N^{1/4}} \right) & \leq & \frac{C_1(i)'}{N^{1/4}} + \P\left( \|\tilde{X}^N(T_{i-1}^N) -\tilde{x}(t_{i-1})\|_1 > \frac{b_i}{N^{1/4}}\right) \nonumber\\
& \le & \frac{C_1(i)'+a_i}{N^{1/4}} \label{z11}
\end{eqnarray}
and
\begin{equation}
\label{z11b}
\P( \sigma_i^N > \bar{s}_i ) \le \frac{C_1(i)'+a_i}{N^{1/4}}
\end{equation}
for $N\ge N_1(b_i,i)+N_2(i)$, the constant $\bar{s}_i$ being defined in \eqref{panoramix}.

On the other hand, if $|\sigma_i^N-s_i|>C_1(b_i,i)/(\eta_i N^{1/4})$, we have either $\sigma_i^N>\bar{s}_i$ or $\sigma_i^N \le \bar{s}_i$ and we deduce from \eqref{panoramix} that
$$
|\tilde{x}_i(t_{i-1}+\sigma_i^N)| = |\tilde{x}_i(t_{i-1}+\sigma_i^N)-\tilde{x}_i(t_{i-1}+s_i)| = \left| \int_{\sigma_i^N}^{s_i} \frac{d\tilde{x}_i}{dt}(t)\ dt \right| \ge \eta_i\ \left| \sigma_i^N - s_i \right| > \frac{C_1(b_i,i)}{N^{1/4}}\,,
$$
so that
$$
\left\{ |\sigma_i^N-s_i| > \frac{C_1(b_i,i)}{\eta_i N^{1/4}}\right\} \subset 	\left\{\sigma_i^N>\bar{s}_i \right\}\cup\left\{ |\tilde{X}^N_i(T_i^N) - \tilde{x}_i(t_{i-1} +\sigma_i^N)| >  \frac{C_1(b_i,i)}{N^{1/4}}  \right \}
$$
since $\tilde{X}_i^N(T_i^N)=0$. We then infer from \eqref{z11}, \eqref{z11b}, and the above inclusion that, for $N\ge N_1(b_i,i)+N_2(i)$,
\begin{equation}
\label{z12}
\P\left(  |\sigma_i^N-s_i| > \frac{C_1(b_i,i)}{\eta_i N^{1/4}}  \right) \leq 2\ \frac{(C_1(i)'+a_i)}{N^{1/4}}\,.
\end{equation}
This estimate now allows us to handle the second term in the right-hand side of \eqref{z10}. Indeed, by Proposition~\ref{prb1}, if $\sigma_i^N\le\bar{s}_i$,
$$
\| \tilde{x}(t_{i-1}+\sigma_i^N)-\tilde{x}(t_{i}) \|_1 \le |\sigma_i^N-s_i| \ \sup_{t\in [t_{i-1},t_{i-1}+\bar{s}_i]} \left\| \frac{d\tilde{x}}{dt}(t) \right\|_1 \le C_4(i)\ |\sigma_i^N-s_i|\,,
$$
and it follows from \eqref{z11b} and \eqref{z12} that, for $N\ge N_1(b_i,i)+N_2(i)$,
\begin{eqnarray}
 \P\left(  \| \tilde{x}(t_{i-1}+\sigma_i^N)-\tilde{x}(t_{i})\|_1 > \frac{C_1(b_i,i)\ C_4(i)}{\eta_i N^{1/4}}  \right) & \leq & \P\left( \sigma_i^N > \bar{s}_i\right) + \P\left( |\sigma_i^N-s_i| > \frac{C_1(b_i,i)}{\eta_i N^{1/4}}  \right) \nonumber\\
& \le & 3\ \frac{(C_1(i)'+a_i)}{N^{1/4}} \,. \label{z13}
\end{eqnarray}
Setting
\begin{equation}
\label{z14}
a_{i+1}:=4\ (a_i + C_1'(i))\,, \quad b_{i+1}:=2\ \frac{(1+C_4(i))\ C_1(b_i,i)}{\eta_i}\,, \quad N_2(i+1):=N_1(b_i,i)+N_2(i)\,,
\end{equation}
we infer from \eqref{z10}, \eqref{z11}, and \eqref{z13} that, for $N\ge N_2(i+1)$,
\begin{eqnarray*}
\P\left( \| \tilde{X}^N(T_i^N) - x(t_i) \|_1 > \frac{b_{i+1}}{N^{1/4}} \right) & \le & \P\left(  \| \tilde{X}^N(T_i^N)-\tilde{x}(t_{i-1}+\sigma_i^N)\|_1 > \frac{C_1(b_i,i)}{N^{1/4}}  \right) \\
& + & \P\left(  \| \tilde{x}(t_{i-1}+\sigma_i^N)-\tilde{x}(t_{i})\|_1 > \frac{C_1(b_i,i)\ C_4(i)}{\eta_i N^{1/4}}  \right) \\
& \le &  \frac{a_{i+1}}{N^{1/4}}\,,
\end{eqnarray*}
which completes the proof. \qed

\medskip

\begin{cor}\label{cor1}
For all $i\ge 1$, there are positive constants $A_i$, $B_i$, and an integer $N_3(i)$ such that
$$
\P\left(  |T_i^N-t_i| > \frac{B_i}{N^{1/4}} \right)\leq \frac{A_i}{N^{1/4}}\;\;\;\mbox{ for }\;\;\; N\ge N_3(i)\,.
$$
\end{cor}
\noindent\textbf{Proof.} Recalling \eqref{z12} and \eqref{z14}, we have
$$
\P\left(  |\sigma_i^N-s_i| > \frac{b_{i+1}}{N^{1/4}}  \right) \leq \frac{a_{i+1}}{N^{1/4}} \;\;\;\mbox{ for }\;\;\; N\ge N_2(i+1)
$$
and $i\ge 1$. Fix $i\ge 1$ and put
$$
N_3(i) := \max_{1\le j \le i} N_2(j+1)\,, \quad A_i := \sum_{j=1}^i a_{j+1}\,, \quad B_i:=\sum_{j=1}^i b_{j+1}\,.
$$
As
$$
T_i^N - t_i = \sum_{j=1}^i (\sigma_j^N - s_j)\,,
$$
we have
$$
\P\left( |T_i^N-t_i| > \frac{B_i}{N^{1/4}} \right)\leq \sum_{j=1}^i \P\left( |\sigma_j^N-s_j| > \frac{b_{j+1}}{N^{1/4}} \right)\leq \sum_{j=1}^i \frac{a_{j+1}}{N^{1/4}} = \frac{A_i}{N^{1/4}}
$$
as claimed. \qed

\medskip

We are now able to prove Proposition~\ref{prop synthese}.

\noindent\textbf{Proof of Proposition~\ref{prop synthese}.} For $I\ge 1$, consider
$$
\Lambda_I :=  \bigcap_{i=1}^I \left\{  \sup_{0\leq s\leq \sigma_i^N}    \| \tilde{X}^N(T_{i-1}^N+s)-x^{(i)}(t_{i-1}+s)  \|_1\leq \frac{C_1(b_i,i)}{N^{1/4}} \ \mbox{ and } \ |T^N_i-t_i|\leq \frac{B_i}{N^{1/4}}   \right\}\,,
$$
and
$$
N_4(i) := \max_{1\le i \le I} \max{\{ N_1(b_i,i) , N_2(i) , N_3(i) \}}\,.
$$
According to Proposition~\ref{prop1}, Proposition~\ref{prop2} and Corollary~\ref{cor1}, we have for $N\ge N_4(i)$
\begin{eqnarray*}
\P\left( \Lambda_I^c \right) & \le & \sum_{i=1}^I \P\left( \sup_{s\in [0,\sigma_i^N]}    \| \tilde{X}^N(T_{i-1}^N+s)-x^{(i)}(t_{i-1}+s)  \|_1 > \frac{C_1(b_i,i)}{N^{1/4}} \right) + \sum_{i=1}^I \P\left( |T^N_i-t_i|> \frac{B_i}{N^{1/4}} \right) \\
& \le & \sum_{i=1}^I \left( \P\left( \| \tilde{X}^N(T_{i-1}^N)-x^{(i)}(t_{i-1})  \|_1 > \frac{b_i}{N^{1/4}} \right) + \frac{C_1(i)'}{N^{1/4}} \right) + \sum_{i=1}^I \frac{A_i}{N^{1/4}} \\
& \le & \sum_{i=1}^I \frac{a_i+C_1(i)'+A_i}{N^{1/4}}
\end{eqnarray*}
\begin{equation} \label{z15}
 \P(\Lambda_I^c)\leq \frac{C_5(I)}{N^{1/4}} .
 \end{equation}
Consider now $t\geq 0$. In $\Lambda_I\cap \{ T_I^N \ge t \}$, there are $i\in\{1,\ldots, I-1\}$, and $s\in [0,\sigma_i^N)$ such that $t=T_{i-1}^N+s$ and
\begin{eqnarray}
T_{i-1}^N + s & \le & T_{i-1}^N + \sigma_i^N = T_i^N - t_i + t_i \le t-I + \frac{B_i}{N^{1/4}} \le \vartheta_I:= \min{\left\{ 1+t_I , \frac{t_I+t_\infty}{2} \right\}}\,, \\
t_{i-1}+s & \le & t_{i-1}+\sigma_i^N = t_{i-1}-T_{i-1}^N+T_i^N-t_i+t_i \le t_I + \frac{2B_i}{N^{1/4}} \le \vartheta_I \label{z15b}
\end{eqnarray}
for $N\ge N_5(I)$ large enough. Consequently, recalling that $x^{(i)}$ is defined in Proposition~\ref{prop1}, it follows from \eqref{senechal} that, in $\Lambda_I\cap \{ T_I^N \ge t \}$
\begin{eqnarray}
\| \tilde{X}^N(t)-x(t)  \|_1 &\leq & \| \tilde{X}^N(T_{i-1}^N+s)-x^{(i)}(t_{i-1}+s)  \|_1+\| x^{(i)}(t_{i-1}+s)-x(t_{i-1}+s)  \|_1 \nonumber\\
& &   +\| x(t_{i-1}+s)-x(T_{i-1}^N+s)  \|_1 \nonumber\\
&\leq&    \frac{C_1(b_i,i)}{N^{1/4}}  + \| x^{(i)}(t_{i-1}+s)-x(t_{i-1}+s)  \|_1 + \ |T_{i-1}^N-t_{i-1}| \ \sup_{t\in [0,\vartheta_I]} \left\| \frac{dx}{dt}(t) \right\|_1 \nonumber \\
& \le & \frac{C_6(I)}{N^{1/4}} + \| x^{(i)}(t_{i-1}+s)-x(t_{i-1}+s)  \|_1 \label{z16}
\end{eqnarray}
for $N\ge N_5(I)$.

Now, since $0\le s< \sigma_i^N$ in $\Lambda_I\cap \{ T_I^N \ge t \}$, we have the following alternative:
\begin{itemize}
\item[(a)] either $s\le s_i$ and $x^{(i)}(t_{i-1}+s)=x(t_{i-1}+s)$,
\item[(b)] or $s_i< s < \sigma_i^N$ and, for $N\ge N_5(I)$, we infer from Proposition~\ref{prb1}, \eqref{senechal}, \eqref{z15b},  and the identity $x^{(i)}(t_i)=x(t_i)$ that
\begin{eqnarray*}
\| x^{(i)}(t_{i-1}+s)-x(t_{i-1}+s)  \|_1 & \le & \| x^{(i)}(t_{i-1}+s)-x^{(i)}(t_i)  \|_1 + \| x(t_i)-x(t_{i-1}+s)  \|_1 \\
& \le & |s-s_i|\ \left( \sup_{t\in [0,\vartheta_I]} \left\| \frac{dx^{(i)}}{dt}(t) \right\|_1  + \sup_{t\in [0,\vartheta_I]} \left\| \frac{dx}{dt}(t) \right\|_1  \right) \\
& \le & C_7(I)\ | \sigma_i^N-s_i| \\
& \le & C_7(I)\ \left( | T_i^N-t_i| + | T_{i-1}^N-t_{i-1}| \right) \\
& \le  & \frac{C_8(I)}{N^{1/4}}\,.
\end{eqnarray*}
\end{itemize}

Combining \eqref{z16} and the above analysis, we conclude that, in $\Lambda_I\cap \{ T_I^N \ge t \}$,
$$
\| \tilde{X}^N(t)-x(t) \|_1 \le \frac{C_9(I)}{N^{1/4}}
$$
for $N\ge N_5(I)$ and thus
$$
\Lambda_I \subset \left\{\sup_{0\leq t\leq T_I^N}    || \tilde{X}^N(t)-x(t)  ||_1\leq \frac{C_9(I)}{N^{1/4}}\right\}\,.
$$
Proposition~\ref{prop synthese} then follows from \eqref{z15} and the above set inclusion. \qed

\medskip

\noindent\textbf{Proof of Theorem~\ref{thc3}.}
Let $t\in (0,t_\infty)$. There exists $I\ge 1$ such that $t<t_I$. Clearly,
$$
\left\{\sup_{0\leq s\leq t}  \| \tilde{X}^N(s)-x(s) \|_1 > \frac{C_0(I)}{N^{1/4}}\right\}\subset \left\{\sup_{0\leq s\leq T_I^N}   \| \tilde{X}^N(s)-x(s)  \|_1 > \frac{C_0(I)}{N^{1/4}}\right\}\cup \left\{t_I>T_I^N\right\}\,,$$
the constant $C_0(I)$ being defined in Proposition~\ref{prop synthese}. Theorem~\ref{thc3} then follows from Proposition~\ref{prop synthese} and Corollary~\ref{cor1}. \qed

\section{Deterministic maximal existence time}\label{dmet}
\setcounter{thm}{0}
\setcounter{equation}{0}
%
\subsection{Global existence}\label{ge}

\noindent\textbf{Proof of Theorem~\ref{thc2}~(i).}
Recall that we assume that there exists $A_0>0$ such that  for all $i,j\geq 1,$
$$
K(i,j)\le \frac{\ln{(i+1)} \wedge \ln{(j+1)}}{4A_0}\,.
$$
For $t\in [0,t_\infty)$ and $i\ge 1$, we define
$$
\phi_i:=\frac{\ln{(i+1)}}{4A_0} \;\;\;\mbox{ and }\;\;\; M_0(t):= \sum_{j=1}^\infty x_j(t)\,.
$$
For $i\ge 1$ and $t\in (t_{i-1},t_i)$, we infer from the upper bound on $K$ and \eqref{b25} that
$$
0 = \frac{dM_0}{dt}(t) +\sum_{j=i}^\infty K(i,j)\ x_j(t) \le \frac{dM_0}{dt}(t) + \phi_i\ M_0(t)\,.
$$
Integrating with respect to time and using the time continuity of $x$ in $\Xun$ gives
$$
M_0(t_i)\ e^{\phi_i t_i} \ge M_0(t_{i-1})\ e^{\phi_i t_{i-1}} = M_0(t_{i-1})\ e^{\phi_{i-1} t_{i-1}}\ e^{(\phi_i - \phi_{i-1}) t_{i-1}}\,.
$$
Arguing by induction, we conclude that
$$
M_0(t_i)\ e^{\phi_i t_i} \ge M_0(0)\ \prod_{j=1}^{i-1} e^{(\phi_{j+1}-\phi_j) t_j}\,, \quad i\ge 2\,.
$$
By \eqref{c7} we have
$$
M_0(t_i) \le \frac{1}{i}\ \sum_{j=i}^\infty j\ x_j(t_i)  = \frac{1}{i}\,, \quad i\ge 2\,.
$$
Combining the above two estimates gives
\begin{eqnarray}
& & \frac{1}{i}\ e^{\phi_i t_i} \ge M_0(0)\ \prod_{j=1}^{i-1} e^{(\phi_{j+1}-\phi_j) t_j} \nonumber\\
& & \phi_i\ t_i \ge \ln{i} + \sum_{j=1}^{i-1} (\phi_{j+1}-\phi_j) t_j + \ln{(M_0(0))}\,, \quad i\ge 2 \nonumber\\
& & t_i \ge 4A_0\ \frac{\ln{i}}{\ln{(i+1)}} + \frac{1}{\ln{(i+1)}}\ \sum_{j=1}^{i-1} \ln{\left( \frac{j+2}{j+1} \right)}\ t_j + \frac{4A_0}{\ln{(i+1)}}\ \ln{(M_0(0))}\,. \label{percevan}
\end{eqnarray}
In particular, for $I\ge 2$ and $i>I$, we infer from \eqref{percevan} and the monotonicity of $(t_j)_{j\ge 1}$ that
\begin{eqnarray*}
t_i & \ge & 4A_0\ \frac{\ln{i}}{\ln{(i+1)}} + \frac{1}{\ln{(i+1)}}\ \sum_{j=I}^{i-1} \ln{\left( \frac{j+2}{j+1} \right)}\ t_I +  \frac{1}{\ln{(i+1)}}\ \sum_{j=1}^{I-1} \ln{\left( \frac{j+2}{j+1} \right)}\ t_1 \\
& + & \frac{4A_0}{\ln{(i+1)}}\ \ln{(M_0(0))} \\
& \ge & 4A_0\ \frac{\ln{i}}{\ln{(i+1)}} + \frac{\ln{(i+1)}-\ln{(I+1)}}{\ln{(i+1)}}\ t_I + \frac{\ln{(I+1)}-\ln{2}}{\ln{(i+1)}}\ t_1 + \frac{4A_0}{\ln{(i+1)}}\ \ln{(M_0(0))} \,.
\end{eqnarray*}

Assume now for contradiction that $t_\infty<\infty$. We may let $i\to\infty$ in the previous inequality to conclude that $t_\infty \ge 4A_0 + t_I$ for all $I\ge 2$. Letting $I\to\infty$ then implies that $t_\infty \ge 4A_0 + t_\infty$ and a contradiction. Therefore, $t_\infty=\infty$. \qed

\subsection{Finite time blow-up of the minimal size}\label{ftbums}

We actually establish a stronger version of the second assertion of Theorem~\ref{thc2}.

\begin{prop}\label{pre2} Consider a coagulation kernel  $K$ and an initial condition $x_0$ satisfying \eqref{c1} and \eqref{c3}, respectively. Let $x$ be the corresponding solution to the min-driven coagulation equations given in Theorem~\ref{thc1} defined on $[0,t_\infty)$, $t_\infty$ being defined in \eqref{c4}. Assume further that there exist a non-decreasing sequence $(\phi_j)_{j\ge 1}$ of nonnegative real numbers, a non-increasing sequence $(\psi_j)_{j\ge 1}$ of nonnegative real numbers, and $\e>0$ such that
\beqn
\label{gaston}
K(i,j) \ge \phi_i \;\;\mbox{ and }\;\; \phi_i\ (\psi_i - \psi_{i+j})\ge \e \;\;\mbox{ for }\;\; j\ge i\ge 1\,.
\eeqn
Then $t_\infty<\infty$.
\end{prop}

\noindent\textbf{Proof.} For $t\in [0,t_\infty)$, define
$$
M_0(t):= \sum_{j=1}^\infty x_j(t) \;\;\mbox{ and }\;\; M_\psi(t):= \sum_{j=1}^\infty \psi_j\ x_j(t)\,.
$$
Given $i\ge 1$ and $t\in (t_{i-1},t_i)$, it follows from \eqref{c5a} and \eqref{b25} that
\bean
\frac{d}{dt} \left( \frac{M_\psi}{M_0} \right) & = & \frac{1}{M_0}\ \sum_{j=i}^\infty (\psi_{i+j}-\psi_i-\psi_j)\ K(i,j)\ x_j + \frac{M_\psi}{M_0^2}\ \sum_{j=i}^\infty K(i,j)\ x_j \\
& = & \frac{1}{M_0}\ \sum_{j=i}^\infty (\psi_{i+j}-\psi_j+\frac{M_\psi}{M_0}-\psi_i)\ K(i,j)\ x_j\,.
\eean
Owing to the monotonicity of $(\psi_j)_{j\ge 1}$, we have
$$
\psi_{i+j} \le \psi_j \;\;\mbox{and }\;\; \frac{M_\psi}{M_0}\le\psi_i\,, \quad j\ge i\,,
$$
so that \eqref{gaston} entails that
$$
(\psi_{i+j}-\psi_j+\frac{M_\psi}{M_0}-\psi_i)\ K(i,j) \le (\psi_{i+j}-\psi_j+\frac{M_\psi}{M_0}-\psi_i)\ \phi_i\,, \quad j\ge i\,.
$$
Consequently,
\bean
\frac{d}{dt} \left( \frac{M_\psi}{M_0} \right) & \le & \frac{\phi_i}{M_0}\ \sum_{j=i}^\infty (\psi_{i+j}-\psi_j+\frac{M_\psi}{M_0}-\psi_i)\ x_j \\
& \le & \frac{\phi_i}{M_0}\ \left( \sum_{j=i}^\infty \psi_{i+j}\ x_j - M_\psi + \frac{M_\psi}{M_0}\ M_0 -\psi_i\ M_0 \right) \\
& \le & \frac{1}{M_0}\ \sum_{j=i}^\infty \phi_i\ (\psi_{i+j}-\psi_i)\ x_j \\
& \le & -\e\,.
\eean
Consequently,
$$
\left( \frac{M_\psi}{M_0} \right)(t_i) + \e\ (t_i-t_{i-1}) \le \left( \frac{M_\psi}{M_0} \right)(t_{i-1})\,.
$$
Summing the above inequality with respect to $i$ gives
$$
\e\ t_\infty \le \lim_{i\to \infty} \left( \frac{M_\psi}{M_0} \right)(t_i) + \e\ t_\infty \le M_\psi(0)/M_0(0)<\infty
$$
and completes the proof. \qed

\medskip

Let us now give some examples of sequences $(\phi_j)_{j\ge 1}$ which fulfil \eqref{gaston}.
\begin{itemize}
\item if $\phi_j=j^\alpha$ for $j\ge 1$ and some $\alpha>0$, then \eqref{gaston} is fulfilled with $\psi_j=j^{-\alpha}$, $j\ge 1$, and $\e=(1-2^{-\alpha})$.
\item if $\phi_j=\left( \ln{(j+1)} \right)^{1+\alpha}$ for $j\ge 1$ and some $\alpha>0$, then \eqref{gaston} is fulfilled with $\psi_j=\left( \ln{(j+1)} \right)^{-\alpha}$, $j\ge 1$, and $\e=\alpha\ 2^{-1-\alpha}\ \ln{(3/2)}$.
\end{itemize}
In particular, Theorem~\ref{thc2}~(ii) follows by combining the second example above with Proposition~\ref{pre2}.

\section{Finite or infinite stochastic time of the last coalescence event}
\setcounter{thm}{0}
\setcounter{equation}{0}

In this section, we study the boundedness or unboundedness of the expectation of the last coalescence time $T^{X_0}$ defined in \eqref{te} with respect to the initial condition $X_0\in\ell^1_\NN$, the space $\ell_\NN^1$ being defined in \eqref{idefix}, when the coagulation kernel has the special structure \eqref{Kmin}, namely,
\begin{equation*}
K(i,j)=\phi(i)\wedge\phi(j) \;\;\;\mbox{ for some positive increasing function }\;\;\; \phi \,.
\end{equation*}
 To this end, we prove some specific properties of the stochastic min-driven coagulation process for this type of kernel. In fact, a crucial argument in the analysis is that this structure allows us to compare the evolution of the process from an arbitrary initial configuration with that starting from monodisperse initial data (that is, initial data of the form $n \ee_i$ for $n\ge 1$ and $i\ge 1$, $(\ee_i)_{i\ge 1}$ being the canonical basis of $\ell^1$ defined in Section~\ref{secdescrip}).

Before going on, we introduce some notations. If $Z\in\ell_\NN^1$ with $\|Z\|_1=n$, the vector $\left( S_1(Z),\ldots,S_n(Z) \right)\in \NN^n$ denotes the collection of the sizes of the particles encoded by $Z$ sorted in increasing order, that is,
\begin{equation}
\label{haddock}
S_m(Z):=1 \;\;\mbox{ if }\;\; 1\le m \le Z_1\,, \quad S_m(Z):=s \;\;\mbox{ if }\;\; 1+\sum_{j=1}^{s-1} Z_j \le m \le \sum_{j=1}^{s} Z_j \;\;\mbox{ and }\;\; 2 \le s \le n\,.
\end{equation}

Next, given an initial condition $X_0\in\ell_\NN^1$ with $n:=\|X_0\|_1$, let $X$ be the stochastic min-driven coagulation process starting from $X_0$ in Section~\ref{secdescrip} and recall that $T^{X_0}$ is defined by
$$
T^{X_0}=\inf\{t\ge 0 \ : \ \| X(t)  \|_1=1 \}\,.
$$
For $i\ge 1$, we also introduce the time
\begin{equation}
\label{milou}
T_i^{X_0} := \inf\{t>0 \ : \ X_1(t)=\ldots=X_i(t)=0\},
\end{equation}
when particles of size smaller or equal than $i$ have disappeared (note that the time $T_i^N$ defined in \eqref{semaphore} in Section~\ref{sec:csp} corresponds to $T_i^{X_0^N}$ with the notation introduced in \eqref{milou}). In addition, since $X_0$ contains $n$ particles, the stochastic process $X$ undergoes $n-1$ coalescence events between $t=0$ and $T^{X_0}$ and we define $L(m)$ to be the minimal size of $X$ after the $(m-1)^{\hbox{\tiny{th}}}$ coalescence event and before the $m^{\hbox{\tiny{th}}}$ coalescence event, $1\le m\le n-1$. Before the latter event, the rate of coagulation is $(n-m)\phi(L(m))$ since $K$ satisfies $K(i,j)=\phi(i)\wedge \phi(j)$. Consequently,
\begin{equation}
\label{tintin}
T^{X_0}=\sum_{m=1}^{n-1}\frac{\e_m}{(n-m)\phi(L(m))},
\end{equation}
where $(\e_m)_{1\le m\le n-1}$ is a sequence of i.i.d. random variables with law $exp(1)$.

\medskip

The first step towards the proof of Theorem~\ref{thc4} is a monotonicity property.

\begin{lem}\label{lemdomi}
Let $X_0$ and $Y_0$ be two initial conditions in $\ell_\NN^1$ such that $\|X_0\|_1=\|Y_0\|_1$ and
\begin{equation}
\label{u6}
S_m(Y_0) \le S_m(X_0)  \ \text{ for all } \ 1\leq m\leq \|X_0\|_1\,.
\end{equation}
Then,  we can construct the stochastic min-driven coagulation processes starting from $X_0$ and $Y_0$ on the same probability space such that $T_i^{X_0} \le T_i^{Y_0}$ for all $i\ge 1$  and  $T^{X_0} \le T^{Y_0}$.
In particular, for all initial data $X_0\in\ell_\NN^1$,
$$
T_1^{X_0} \le T_1^{\|X_0\|_1 \ee_1} \;\;\;\mbox{ and }\;\;\; T^{X_0}\le T^{\|X_0\|_1 \ee_1}\,.
$$
\end{lem}

\begin{proof}
Let $X$ and $Y$ denote the stochastic min-driven coagulation processes starting from $X_0$ and $Y_0$, respectively, and define $n:=\|X_0\|_1  = \|Y_0\|_1$. Between $t=0$ and $T^{X_0}$, the process $X$ reaches $n$ different states $\left\{ \hat{X}(j) \ : \ 1\le j \le n-1 \right\}$ with $\hat{X}(0)=X_0$ and $\|\hat{X}(j)\|_1=n-j$. In other words, $\hat{X}(j)$ is the state of $X$ after the $j^{\hbox{\tiny{th}}}$ coalescence event and actually denotes $X(\theta_j)$, $\theta_j$ being the time at which the $j^{\hbox{\tiny{th}}}$ coalescence event occurs. Analogously, between $t=0$ and $T^{Y_0}$, the process $Y$ reaches $n$ different states $\left\{ \hat{Y}(j) \ : \ 1\le j \le n-1 \right\}$ with $\hat{Y}(0)=Y_0$ and $\|\hat{Y}(j)\|_1=n-j$.

We first prove by induction that we can construct the processes $X$ and $Y$ on the same probability space such that
\begin{equation}
\label{u6b}
S_m\left( \hat{Y}(j) \right) \le S_m\left( \hat{X}(j) \right)\,, \quad 1\le m \le n-j\,, \quad 1 \le j \le n-1\,.
\end{equation}
Owing to \eqref{u6}, this inequality is clearly fulfilled for $j=0$. Assume now that \eqref{u6b} holds true for some $j\in\{0,\ldots,n-2\}$ and set
$$
S_m^{X,j}:=S_m\left( \hat{X}(j) \right) \;\;\;\mbox{ and }\;\;\; S_m^{Y,j}:=S_m\left( \hat{Y}(j) \right)\,, \quad 1\le j \le n-i\,.
$$
Since the coagulation kernel $K$ is of the form \eqref{Kmin}, we may couple the two processes $X$ and $Y$ in such a way that $\hat{X}(j+1)$ is obtained by coalescing the particles of sizes $S_1^{X,j}$ and $S_k^{X,j}$ and $\hat{Y}(j+1)$ by coalescing the particles of sizes $S_1^{Y,j}$ and $S_k^{Y,j}$ with the same index $k$ chosen in $\{2,\ldots,n-i\}$ with uniform law.
Thus,
\begin{eqnarray*}
\left\{ S_m\left( \hat{X}(j+1) \right) \ : \ 1 \le m \le n-j-1 \right\} & = & \left\{ S_2^{X,j}, \ldots, S_{k-1}^{X,j}, S_{k+1}^{X,j},\ldots,S_{n-j}^{X,j} \right\} \cup \left\{ S_1^{X,j} + S_k^{X,j} \right\} \,,\\
\left\{ S_m\left( \hat{Y}(j+1) \right) \ : \ 1 \le m \le n-j-1 \right\} & = & \left\{ S_2^{Y,j}, \ldots, S_{k-1}^{Y,j}, S_{k+1}^{Y,j},\ldots,S_{n-j}^{Y,j} \right\} \cup \left\{ S_1^{Y,j} + S_k^{Y,j} \right\} \,.
\end{eqnarray*}
At this stage, the inequality \eqref{u6b} is not obvious as the reordering of the sizes can be different in $\hat{X}(j+1)$ and $\hat{Y}(j+1)$. The situation can be represented as follows:
$$
\begin{array}{lllclcll}
S_1^{Y,j} & \le \ldots \le & S_{k-1}^{Y,j} \leq \ldots \leq & S_1^{Y,j} +S_k^{Y,j} & \leq \ldots \leq & \ldots & \leq \ldots \leq & S_{n-i}^{Y,j}\,, \\
 & & & & & & & \\
S_1^{X,j} & \le \ldots \le & S_{k-1}^{X,j} \leq \ldots \leq & \ldots & \leq \ldots \leq & S_1^{X,j} + S_k^{X,j} & \leq \ldots \leq & S_{n-i}^{X,j}\,.
\end{array}
$$
Nevertheless, we observe that
$$ S_m\left( \hat{Y}(j+1) \right)\left\{
          \begin{array}{cl}
          S_{m+1}^{Y,j}   &   \text{for } 1\leq m\leq k-2\,,\\
          & \\
         \max{\left\{ \min{\left\{ S_{m+2}^{Y,j} , S_1^{Y,j} + S_k^{Y,j} \right\}} , S_{m+1}^{Y,j} \right\}}   &   \text{for } m\geq k-1\,,\\
        \end{array}
 \right.
 $$
 and
 $$ S_m\left( \hat{X}(j+1) \right)\left\{
          \begin{array}{cl}
          S_{m+1}^{X,j}   &   \text{for } 1\leq m\leq k-2\,,\\
          & \\
         \max{\left\{ \min{\left\{ S_{m+2}^{X,j} , S_1^{X,j} + S_k^{X,j} \right\}} , S_{m+1}^{X,j} \right\}}   &   \text{for } m\geq k-1\,,\\
        \end{array}
 \right.
 $$
from which \eqref{u6b} for $j+1$ readily follows thanks to \eqref{u6b} for $j$.

We next claim that the random number of coalescence events needed to exhaust the particles of size $i\ge 1$ is smaller for $X$ than for $Y$, that is, \begin{equation}
\label{u6c}
n_i^{X_0}\le n_i^{Y_0}\,, \qquad i\ge 1\,,
\end{equation}
where
\begin{eqnarray*}
n_i^{X_0} & := &  \inf{\left\{ j\in\{ 0,\ldots,n-1\}\ : \ S_1\left( \hat{X}(j) \right)\ge i+1 \right\}}\,, \\
n_i^{Y_0} & := &  \inf{\left\{ j\in\{ 0,\ldots,n-1\}\ : \ S_1\left( \hat{Y}(j) \right)\ge i+1 \right\}}\,.
\end{eqnarray*}
Indeed, we have $S_1\left( \hat{Y}(j) \right)\le S_1\left( \hat{X}(j) \right)\le i$ for $1\le j \le n_i^{X_0}-1$ by \eqref{u6b}.

We can now prove the lemma. For $i\ge 1$, we have
$$
T_i^{X_0} = \sum_{j=1}^{n_i^{X_0}} \frac{\varepsilon_j}{(n-j) \phi\left( S_1\left( \hat{X}(j-1) \right) \right)} \;\;\mbox{ and }\;\;
T_i^{Y_0} = \sum_{j=1}^{n_i^{Y_0}} \frac{\varepsilon_j}{(n-j) \phi\left( S_1\left( \hat{Y}(j-1) \right) \right)}\,,
$$
where $(\varepsilon_k)_{k\ge 1}$ is a sequence of i.i.d. random variables with law $exp(1)$. Concerning $T^{X_0}$ and $T^{Y_0}$, we have
 $$
T^{X_0} = \sum_{j=1}^{n-1} \frac{\varepsilon_j}{(n-j) \phi\left( S_1\left( \hat{X}(j-1) \right) \right)} \;\;\mbox{ and }\;\;
T^{Y_0} = \sum_{j=1}^{n-1} \frac{\varepsilon_j}{(n-j) \phi\left( S_1\left( \hat{Y}(j-1) \right) \right)}\,.
$$

 The expected result then follows by \eqref{u6b}, \eqref{u6c}, and the monotonicity of $\phi$.
\end{proof}

\medskip

We next prove that the expectation of the time $T_1^{X_0}$ after which all particles of size $1$ have disappeared is bounded independently of the initial condition $X_0$ (as soon as $X_0\neq \ee_1$). According to Lemma~\ref{lemdomi}, it will be sufficient to prove such a bound for monodisperse initial data of the form $n \ee_1$, $n\ge 2$.

\begin{lem}\label{lemmajesp}
There exists $C>0$ such that, for any initial condition $X_0\in\ell_\NN^1$ with $X_0\neq \ee_1$,
$$\E(T_1^{X_0})\le C\,,$$
the time $T_1^{X_0}$ being defined in \eqref{milou}.
\end{lem}

\begin{proof} Let $n:=\|X_0\|_1$ be the initial number of particles. If $n=1$ and $X_0\neq \ee_1$, then $T_1^{X_0}=0$. So,  we assume that $n\ge 2$. By Lemma~\ref{lemdomi}, we have the stochastic domination $T_1^{X_0} \le T_1^{n\ee_1}$, so that
\begin{equation}
\label{u10}
\E(T_1^{X_0})\le \E(T_1^{n\ee_1})\,,
\end{equation}
and it suffices to obtain an upper bound on $\E(T_1^{n\ee_1})$ which does not depend on $n\ge 2$.

We consider  the solution $x$ to the deterministic min-driven coagulation equation \eqref{c5a} with monodisperse initial condition $x_0=(x_{i,0})_{i\ge 1}$ given by $x_{1,0}=1$ and $x_{i,0}=0$ for $i\ge 2$. It follows from Corollary~\ref{cor1} that
$$
\P\left(  |T_1^{n\ee_1}-t_1| > \frac{B_1}{n^{1/4}} \right) \leq \frac{A_1}{n^{1/4}}\,, \quad n\ge N_3(1)\,, $$
from which we deduce that there is $C>0$ such that
\begin{equation}\label{equat1}
\P(T_1^{n\ee_1} > B_1+t_1) \le \frac{C}{n^{1/4}}\,, \quad n\ge 2\,.
\end{equation}
Introducing the (random) number of coalescence events $n_1$ performed between $t=0$ and $T_1^{n\ee_1}$, we have
$$
T_1^{n\ee_1}=\sum_{m=1}^{n_1} \frac{\e_m}{(n-m)\phi(1)}\,,
$$
where $(\e_m)_{1\le m\le n-1}$ is a sequence of i.i.d. random variables with law $exp(1)$. Obviously, $n_1\le n-1$ which gives the bound
$$
T_1^{n\ee_1}\leq\frac{1}{\phi(1)}\sum_{m=1}^{n-1} \frac{\e_m}{m}\,.
$$
Since $\E(\e_m)=1$ and $\E(\e_m^2)=2$ for $1\le m\le n$, we deduce from \eqref{equat1}, the H\"older inequality, and the above estimate that
\begin{eqnarray*}
\E\left( T_1^{n\ee_1} \right)&=&\E\left( T_1^{n\ee_1}\ \mathds{1}_{[0,B_1+t_1]}(T_1^{n\ee_1}) \right) + \E\left( T_1^{n\ee_1}\ \mathds{1}_{(B_1+t_1,\infty)}(T_1^{n\ee_1}) \right)\\
&\le & B_1+t_1+ \frac{1}{\phi(1)}\ \sum_{m=1}^{n-1} \frac{1}{m}\ \E\left( \e_m\ \mathds{1}_{(B_1+t_1,\infty)}(T_1^{n\ee_1}) \right) \\
& \le & B_1+t_1 + \frac{1}{\phi(1)}\ \sum_{m=1}^{n-1} \frac{1}{m}\ \E\left( \e_m^2 \right)^{1/2}\ \P\left( T_1^{n\ee_1}> B_1+t_1 \right)^{1/2}\\
&\le & B_1+t_1+ \frac{C}{\phi(1) n^{1/8}} \ \sum_{m=1}^{n-1} \frac{1}{m} \\
& \le & B_1+t_1+C\ \frac{\ln{n}}{n^{1/8}}\,.
\end{eqnarray*}
Since $B_1$ and $t_1$ do not depend on $n$ (actually one has $t_1=1/\phi(1)$), we have established the expected upper bound from which Lemma~\ref{lemmajesp} follows by \eqref{u10}.
\end{proof}

\medskip

The next step is to establish a connection between the early stages of the dynamics of the processes starting from monodisperse initial data.

\begin{lem}\label{lemmegaloi}
For $n\ge 2$ and $i\ge 1$ we have
$$
T_i^{n\ee_i}\;\overset{\hbox{\small{law}}}{=}\;\frac{\phi(1)}{\phi(i)}\ T_1^{n\ee_1}\,.
$$
\end{lem}

\begin{proof}
As in the proof of Lemma~\ref{lemdomi}, a coupling can be done between the processes starting from $n\ee_1$ and $n\ee_i$ so that
$$
T_1^{n\ee_1} = \sum_{m=1}^{n_1}\frac{\e_m}{(n-m)\phi(1)} \;\;\;\mbox{ and }\;\;\; T_i^{n\ee_i} = \sum_{m=1}^{n_1} \frac{\e_m}{(n-m)\phi(i)}
$$
with the same random number of coalescence events $n_1$ and sequence $(\e_m)_{1\le m\le n-1}$ of i.i.d. random variables with law $exp(1)$ for both processes.
\end{proof}

\medskip

\begin{proof}[Proof of Theorem~\ref{thc4}] Assume first that
$$
\sum_{i=1}^\infty \frac{1}{i\phi(i)}<\infty\,.
$$
Thanks to Lemma~\ref{lemdomi}, we just have to show that $\E(T^{n\ee_1})$ is bounded independently of $n\ge 1$.

To this end, we fix $n\ge 1$. Let us first notice that, if $n=1$, then $T^{n\ee_1}=0$. Assume now that $n\ge 2$ and for $i\ge 1$, let $X$ be the stochastic min-driven coagulation process starting from $n\ee_i$. Clearly, $T_j^{n\ee_i}=0$ for $1\le j \le i-1$ and we define the (random) number $n_\ast:=\|X(T_i^{n\ee_i})\|_1$ of particles in the system at time $T_i^{n\ee_i}$ and $Y:=X(T_i^{n \ee_i})$. Notice that $Y_j=X_j(T_i^{n\ee_i})=0$ for $1\le j\le 2i-1$ and the conservation of mass warrants that $n_\ast\le n/2$ as
$$
2i\ n_\ast = 2i\ \|X(T_i^{n\ee_i})\|_1 \le \|X(T_i^{n\ee_i})\|_{1,1} = \|n\ee_i\|_{1,1} = n i\,.
$$
Moreover, the properties of $Y$ and Lemma~\ref{lemdomi} yield the stochastic domination $T^Y \le T^{n_\ast \ee_{2i}}$.
Since
$$
T^{n \ee_i}\;\overset{\hbox{\small{law}}}{=}\;T_i^{n\ee_i} + T^Y\,,
$$
where, conditionally on $Y$, $T_i^{n\ee_i}$ and $T^Y$ are independent, it follows from Lemma~\ref{lemmegaloi} that
\begin{equation} \label{quiquidonc}
T^{n\ee_i}\le \frac{\phi(1)}{\phi(i)}\ T_1^{n\ee_1}+T^{n_\ast \ee_{2i}}\,.
\end{equation}

Let us now prove by induction on $n$ that the property 
$$ 
\mathcal{P}(n): \ \ \ \ \ \ \  \E(T^{n\ee_{2^i}})\le C \sum_{j=i}^\infty \frac{\phi(1)}{\phi(2^j)} \;\;\;\mbox{ for all }\;\; i\ge 0 \;\;\mbox{ and }\;\; 0\le m \le n\,, 
$$ 
holds true for all $n\ge 0$, where $C$ is the constant appearing in Lemma~\ref{lemmajesp}.

It is clear for $n=0$ . Consider$ n\geq 1$ and assume $\mathcal{P}(n-1)$. For $i\ge 0$, it follows from \eqref{quiquidonc} and $\mathcal{P}(n-1)$ that there is $n_{\ast} \leq n/2$ such that 
\begin{eqnarray*}
  \E(T^{n\ee_{2^i}}) &\le&  \frac{\phi(1)}{\phi(2^i)}\ \E(T_1^{n\ee_1}) +\E(T^{n_\ast \ee_{2^{i+1}}}) \\
&\le&  \frac{\phi(1)}{\phi(2^i)}\ \E(T_1^{n\ee_1}) + \sum_{m=1}^{n/2}\P(n^*=m)\ \E(T^{m \ee_{2^{i+1}}})   \\
&\le&  \frac{\phi(1)}{\phi(2^i)}\ \E(T_1^{n\ee_1}) + \sup_{1\le m\le n/2}\E(T^{m \ee_{2^{i+1}}})   \\
&\le&  \frac{\phi(1)}{\phi(2^i)}\ \E(T_1^{n\ee_1}) + C \sum_{j=i+1}^\infty \frac{\phi(1)}{\phi(2^j)} \ \ \ \ \ (\text{by induction hypothesis}) \\
&\le&   C \sum_{j=i}^\infty \frac{\phi(1)}{\phi(2^j)}\,,
\end{eqnarray*}
which proves $\mathcal{P}(n)$.

We then infer from Property~$\mathcal{P}(n)$ for $i=0$ that
$$
\E(T^{n\ee_1})\le C\phi(1)\ \sum_{i=0}^\infty \frac{1}{\phi(2^i)} < \infty\,,
$$
the convergence of the series $\sum 1/\phi(2^i)$ being ensured by that of $\sum 1/(i\phi(i))$ and the monotonicity of $\phi$.

\medskip

To prove the converse part of Theorem~\ref{thc4}, we assume that
$$
\sum_{i=1}^\infty \frac{1}{i\phi(i)}=\infty\,,
$$
and show that, for each constant $C>0$,  there exists a configuration $X_0$ such that $\E(T^{X_0}) \geq C$. More precisely, we will prove that
\begin{equation}\label{Tninfini}
\lim_{n\rightarrow \infty} \E(T^{n\ee_1})=\infty.
\end{equation}
Indeed, let $n\ge 2$. By \eqref{tintin}, we have
$$
T^{n\ee_1}=\sum_{m=1}^{n-1}\frac{\e_m}{(n-m)\phi(L(m))}\,,
$$
where $(\e_m)_{1\le m\le n-1}$ is a sequence of i.i.d. random variables with law $exp(1)$. The sequence $(L(m))_{1\le m \le n-1}$ is random but let us notice the bound
$$
L(m)\le \frac{n}{n-m+1}\le \frac{n}{n-m}\,, \quad 1 \le m \le n-1\,,$$
which follows from the conservation of mass since there remain $n-m+1$ particles in the system before the $m^{\hbox{\tiny{th}}}$ coalescence event. Therefore, thanks to the monotonicity of $\phi$,
$$
\E(T^{n\ee_1})\ge \sum_{m=1}^{n-1} \frac{1}{m\phi(n/m)}\,,$$
and the divergence of the series $\sum 1/(i\phi(i))$ ensures that
$$
\lim_{n\rightarrow \infty} \sum_{m=1}^{n-1}\frac{1}{m\phi(n/m)}=\int_{1}^\infty \frac{dx}{x\phi(x)}=\infty\,,$$
which completes the proof.
\end{proof}


\end{document}